\documentclass[11pt,reqno]{amsart}

\usepackage{amsmath,amssymb,amsthm,bbm}

\usepackage[tmargin=1.25in,bmargin=1.25in,lmargin=1.65in,rmargin=1.65in]{geometry} 
\usepackage{mathtools}
\usepackage{xspace}
\usepackage{wrapfig}
\usepackage{xfrac}
\usepackage{verbatim}
\usepackage{thmtools, thm-restate}
\usepackage{subcaption}

\usepackage{soul}
\usepackage{microtype}

\usepackage[usenames]{color}

\newcommand{\black}{\black}

\usepackage[font=small,labelfont=bf]{caption}
\usepackage[bookmarks, colorlinks, breaklinks,
pdftitle={},pdfauthor={}]{hyperref}

\hypersetup{
  linkcolor=black,
  citecolor=black,
  filecolor=black,
  urlcolor=black}

\theoremstyle{plain}
\newtheorem{theorem}{Theorem}[section]
\newtheorem{proposition}[theorem]{Proposition}
\newtheorem{corollary}[theorem]{Corollary}
\newtheorem{lemma}[theorem]{Lemma}
\newtheorem{definition}{Definition}
\newtheorem{example}[theorem]{Example}
\newtheorem{hypothesis}[theorem]{Hypothesis}

\numberwithin{equation}{section}

\newenvironment{claim}[1]{\par\noindent\emph{Claim}:\space#1}{}
\newenvironment{claimproof}[1]{\par\noindent\emph{Proof}:\space#1}{\leavevmode\unskip\penalty9999 \hbox{}\nobreak\hfill\quad\hbox{$\blacksquare$}}

\usepackage{float}
\usepackage{cleveref}
\usepackage{tikz}

\crefname{hypothesis}{Hypothesis}{Hypothesis} % always uppercase       

\usepackage{booktabs}

\usetikzlibrary{calc,decorations.markings}
\usetikzlibrary{decorations.pathmorphing}
\usetikzlibrary{decorations.pathreplacing}
\usetikzlibrary{patterns}
\usetikzlibrary{arrows}

\DeclareFontFamily{U}{matha}{\hyphenchar\font45}
\DeclareFontShape{U}{matha}{m}{n}{
      <5> <6> <7> <8> <9> <10> gen * matha
      <10.95> matha10 <12> <14.4> <17.28> <20.74> <24.88> matha12
      }{}
\DeclareSymbolFont{matha}{U}{matha}{m}{n}
\DeclareFontSubstitution{U}{matha}{m}{n}

\DeclareFontFamily{U}{mathx}{\hyphenchar\font45}
\DeclareFontShape{U}{mathx}{m}{n}{
      <5> <6> <7> <8> <9> <10>
      <10.95> <12> <14.4> <17.28> <20.74> <24.88>
      mathx10
      }{}
\DeclareSymbolFont{mathx}{U}{mathx}{m}{n}
\DeclareFontSubstitution{U}{mathx}{m}{n}

\DeclareMathDelimiter{\vvvert}{0}{matha}{"7E}{mathx}{"17}

\newcommand{\bb}[1]{\mathbb{#1}}
\newcommand{\cc}[1]{\mathcal{#1}}

\newcommand{\co}[1]{\left[#1\right )} %% Closed then open braces
\newcommand{\oc}[1]{\left(#1\right ]} %% Open then closed braces
\newcommand{\ob}[1]{\left(#1\right )} %% Open braces
\newcommand{\cb}[1]{\left[#1\right ]} %% Closed braces
\newcommand{\abs}[1]{\left\vert#1\right\vert} %% Absolute value of argument
\newcommand{\norm}[1]{\|#1\|} %% Norm of argument
\newcommand{\ceil}[1]{\lceil #1\rceil}
\newcommand{\mnorm}[1]{\vvvert #1\vvvert}

\newcommand{\indicator}{\mathbbm{1}}
\newcommand{\indicatorthat}[1]{{\mathbbm{1}}_{\left\{#1\right\}}}

\newcommand{\V}[1]{\boldsymbol{#1}}
\newcommand{\bydef}{\coloneqq}
\newcommand{\nrp}{k_{0}}

\newcommand{\R}{\bb R}
\newcommand{\Z}{\bb Z}

\newcommand{\N}{\bb N}
\renewcommand{\P}{\bb P}

\newcommand{\PP}{\mathsf{P}}
\newcommand{\adj}{\mathsf{adj}}
\newcommand{\slb}{p_{1}}
\newcommand{\flipcost}{\slb^2}
\newcommand{\tflipcost}{\mathrm{flip}_{\kappa}} % total flip cost.

\newcommand{\flip}{\cc F} % Flip operator
\newcommand{\trans}{\cc T} % translation operator
\newcommand{\refl}{\cc R} % reflection operator

\newcommand{\saw}{\Gamma}%{\mathsf{SAW}}
\newcommand{\sap}{\tilde{\saw}}

\newcommand{\spa}{\mathrm{span}}
\newcommand{\bridge}{\mathsf{B}} % Bridges
\newcommand{\hsw}{\mathsf{H}} % Half-space walks.
\newcommand{\bp}{\tau} % Bridge point
\newcommand{\walk}{\mathsf{W}} % set of all walks
\newcommand{\sawno}{\saw_{n}^{o}}

\newcommand{\wt}{\cc W}
\newcommand{\srwtwo}{S} % SRW 2-point function.
\newcommand{\graph}{\mathsf{G}} % set of graphs
\newcommand{\graphc}{\graph^{\mathsf{c}}} % set of connected graphs
\newcommand{\lace}{\mathsf{L}} % set of lace graphs
\newcommand{\lacemap}{\cc L} % lace map
\newcommand{\comp}{\mathsf{C}} % compatible edges

\newcommand{\vertex}{r}

\newcommand{\HHSpi}{\tilde\Pi}

\newcommand{\con}{a}

\allowdisplaybreaks

\date{Dec.\ 8, 2018}
\title{Self-attracting self-avoiding walk}
\author{Alan Hammond and Tyler Helmuth}
\address{Departments of Mathematics and Statistics\\ 
  U.C.\ Berkeley, Berkeley, CA, 94720-3840 USA}
\email{alanmh@stat.berkeley.edu}
\address{Department of Mathematics\\ 
  University of Bristol, Bristol, UK, BS8 1TW}
\email{jhelmt@gmail.com}

\begin{document}
\maketitle

\begin{abstract}
  This article is concerned with self-avoiding walks (SAW) on $\Z^{d}$
  that are subject to a self-attraction. The attraction, which rewards
  instances of adjacent parallel edges, introduces difficulties that
  are not present in ordinary SAW. Ueltschi has shown how to overcome
  these difficulties for sufficiently regular infinite-range step
  distributions and weak self-attractions~\cite{Ueltschi}. This
  article considers the case of bounded step distributions. For weak
  self-attractions we show that the connective constant exists, and,
  in $d\geq 5$, carry out a lace expansion analysis to prove the
  mean-field behaviour of the critical two-point function, hereby
  addressing a problem posed by den Hollander~\cite{dH}.

  \smallskip
  \smallskip
  \noindent \textbf{\keywordsname.} Self-interacting random walk,
  self-attracting walk, self-avoiding walk, linear polymers, lace expansion,
  critical phenomena, Hammersley-Welsh argument.
\end{abstract}

\section{Introduction}
\label{sec:introduction}

\subsection{Model definition}
\label{sec:model}

Let $\Z^{d}$ denote the $d$-dimensional integer lattice with
nearest-neighbour edges, and assume $d\geq 2$.  Let $\P$ be the law of
a random walk on the vertices of $\Z^{d}$ with i.i.d.\ increments
distributed according to a step distribution $D$. Letting
$\{\pm e_{i}\}_{i=1}^{d}$ denote the standard generators of $\Z^{d}$,
a \emph{plaquette} is a collection of vertices of the form
$\{x,x+u,x+v,x+u+v\}$ where $u\notin\{\pm v\}$ and $u$ and $v$
  are in $\{\pm e_{i}\}_{i=1}^{d}$. Two edges $\{x_{1},y_{1}\}$ and
$\{x_{2},y_{2}\}$ of $\Z^{d}$ are \emph{adjacent} if
$\{x_{1},y_{1},x_{2},y_{2}\}$ is a plaquette.

A walk is a sequence of vertices in $\Z^{d}$, and the \emph{edges} of
a walk $\omega$ are the pairs $\{\omega_{i},\omega_{i+1}\}$ of
consecutive vertices. Note that, for general increment
distributions~$D$, the edges of a walk in the support of $\P$ are not
necessarily edges of~$\Z^{d}$. Define $\adj(\omega)$ to be the
collection of pairs of edges of $\omega$ that are adjacent edges of
$\Z^{d}$, and let $\abs{\adj(\omega)}$ be the cardinality of this
set. See \Cref{fig:Walk-Example}.
  
Let $\P_{n}$ denote the law induced by $\P$ on $n$-step walks that
begin at the origin $o\in \Z^{d}$, and recall that a walk is
  \emph{self-avoiding} if it does not visit any vertex more than once
  (see \Cref{sec:conventions} for a more precise definition).  The
models we are interested in are perturbations $\PP_{n,\kappa}$ of
$\P_{n}$ defined by
\begin{equation}
  \label{eq:Model}
  \PP_{n,\kappa}(\omega) 
  \propto \indicatorthat{\omega\in \saw_{n}} \wt_{\kappa}(\omega),
  \qquad \kappa\geq 0,
\end{equation}
where $\saw_{n}$ is the set of $n$-step self-avoiding walks with
initial vertex $\omega_{0}=o$ and
\begin{equation}
  \label{eq:Hamiltonian}
  \wt_{\kappa}(\omega)  \bydef e^{-H_{\kappa}(\omega)}\P_{n}(\omega),\qquad
  e^{-H_{\kappa}(\omega)} \bydef (1+\kappa)^{\abs{\adj(\omega)}}.
\end{equation}
The symbol $\bydef$ indicates equality by definition. The law
$\PP_{n,\kappa}$ on $n$-step walks is called \emph{($n$-step)
  attracting self-avoiding walk} with \emph{attraction strength
  $\kappa$}, or \emph{($n$-step) $\kappa$-ASAW}. When the length of
the walk is irrelevant the adjective $n$-step will be dropped. We
think of the right-hand side of~\eqref{eq:Model} as defining the
$\kappa$-ASAW \emph{weight} of a walk $\omega$. The probability of a
walk is proportional to its weight.

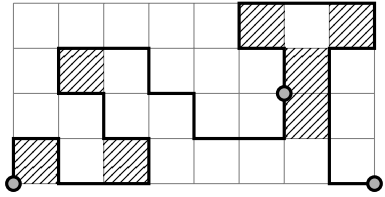
\begin{figure}[]
  \centering
  \begin{tikzpicture}[scale=.6,
    dot/.style={draw,circle,black, fill=black, minimum size=.05}]
    \draw[gray] (0,0) grid (8,4);
    \draw[black, very thick] (0,0) -- (0,1) -- (1,1) -- (1,0) -- (3,0)
    -- (3,1) -- (2,1) 
    -- (2,2) -- (1,2) -- (1,3) -- (3,3) -- (3,2) -- (4,2) -- (4,1) -- (6,1) -- (6,3) --
    (5,3) -- (5,4) -- (8,4) -- (8,3) -- (7,3) -- (7,0) -- (8,0);
    \fill[pattern= north east lines] (1,2) rectangle (2,3);
    \fill[pattern= north east lines] (0,0) rectangle (1,1);
    \fill[pattern= north east lines] (2,0) rectangle (3,1);
    \fill[pattern= north east lines] (6,1) rectangle (7,2);
    \fill[pattern= north east lines] (6,2) rectangle (7,3);
    \fill[pattern= north east lines] (5,3) rectangle (6,4);
    \fill[pattern= north east lines] (7,3) rectangle (8,4);
    \draw[black,very thick, fill=black!30] (0,0) circle [radius=.15];
    \draw[black,very thick,fill=black!30] (8,0) circle [radius=.15];
    \draw[black,very thick,fill=black!30] (6,2) circle [radius=.15];
  \end{tikzpicture}
  \caption{A self-avoiding walk $\omega$. Shaded plaquettes indicate
    the seven pairs of adjacent edges of $\omega$.}
  \label{fig:Walk-Example}
\end{figure} 

The law $\PP_{n,0}$ defined by~\eqref{eq:Model} is the law of $n$-step
\emph{self-avoiding walk} (SAW)~\cite{BCS}. Physically, self-avoiding walk
is a model of a linear polymer in a good solvent. The self-avoidance
constraint represents the inability of two molecules in the polymer to
occupy the same space. If $\kappa>0$, walks under the $\kappa$-ASAW
law are attracted to themselves. Physically, this is a model of a
linear polymer in a poor
solvent, see~\cite[Section~6.3]{Slade} and~\cite[Chapter~6]{dH}. The molecules huddle together to escape exposure to the surrounding solvent.

\subsection{Lack of submultiplicativity}
\label{sec:lack-subm}

Before stating our results, we briefly discuss the central difficulty
of the model. Let $c_{n}(\kappa)$ denote the normalization constant
that makes $\PP_{n,\kappa}$ a probability measure, i.e.,
\begin{equation}
  \label{eq:cn}
  c_{n}(\kappa) \bydef \sum_{\omega\in\saw_{n}} \wt_{\kappa}(\omega). 
\end{equation}
Note that $c_{n}(\kappa)$ is implicitly also a function of the step
distribution $D$.

The first mathematical fact one learns about self-avoiding walk is
that when $\kappa=0$ the sequence $(c_{n}(\kappa))_{n\geq 1}$ is
submultiplicative, i.e.,
\begin{equation*}
  c_{n+m}(0)\leq c_{n}(0)c_{m}(0).
\end{equation*}
This bound arises because any $(n+m)$-step self-avoiding walk can be
split into an $n$-step self-avoiding walk and an $m$-step
self-avoiding walk.  Simple estimates and Fekete's
lemma~\cite[Lemma~1.2.1]{Steele} on submultiplicative sequences imply
that $(c_{n}(0))^{1/n}$ converges as $n\to\infty$:
see~\cite[Section~1.2]{MadrasSlade}.\footnote{Note that our definition
  of $c_{n}$ involves $D$, i.e., we are enumerating weighted
  self-avoiding walks.}

The basic difficulty in the study of
$\kappa$-ASAW is that the sequence $(c_{n}(\kappa))_{n\geq
  1}$ is generally not submultiplicative for
$\kappa>0$. To see that submultiplicativity cannot hold in general,
consider the nearest-neighbour step distribution $D(x) =
(2d)^{-1}\indicatorthat{\norm{x}_{1}=1}$.  Submultiplicativity of the
sequence $c_{n}(\kappa)$ would imply
\begin{equation*}
c_{n}(\kappa)\leq c_{1}(\kappa)^{n}=1,
\end{equation*}
which cannot hold for fixed $n\geq 3$ when
$\kappa$ is sufficiently large, as the left-hand side is a polynomial
in $\kappa$ of degree at least $1$. 

\section{Results}
\label{sec:results}

Henceforth it will be assumed that $D(x)$ is invariant under the
symmetries of $\Z^{d}$ (namely, reflections in hyperplanes and
rotations by $\pi/2$ about coordinate axes), and that
$D(e_{1}) \bydef \slb>0$. Thus, $D(\pm e_{i})=\slb$ for both choices of
sign and all choices of $i=1,2,\dots, d$.

The next two subsections present our main results,
\Cref{thm:CC-Exists-Intro} and \Cref{thm:High-D}, and in
\Cref{sec:main-idea,sec:disc} we discuss the main ideas of the proofs
and briefly describe how our results fit into the literature.

\subsection{Connective constants}
\label{sec:connective-constant}

The limiting value $\mu(\kappa)$ of $(c_{n}(\kappa))^{1/n}$, if the
limit exists, is called the \emph{connective constant with
  self-attraction $\kappa$}. We will prove that the connective
constant of $\kappa$-ASAW exists for $\kappa$ sufficiently small
despite not knowing if
submultiplicativity holds.
\begin{theorem}
  \label{thm:CC-Exists-Intro}
  Let $d\geq 2$. There exists a $\kappa_{0} = \kappa_{0}(D)>0$ such that for
  $0<\kappa<\kappa_{0}$ the limit 
  $\mu(\kappa) = \lim_{n\to\infty}(c_{n}(\kappa))^{1/n}$ exists.
\end{theorem}
Note that the dependence of $\kappa_{0}$ on $D$ implicitly means that
$\kappa_{0}$ may depend on the dimension $d$. The remainder of this
section briefly describes the proof of \Cref{thm:CC-Exists-Intro},
although we delay a discussion of how the lack of submultiplicativity
is overcome to \Cref{sec:main-idea}. The proof appears in
\Cref{sec:CC}.

An $n$-step self-avoiding walk $\omega$ is a \emph{bridge} if
$\pi_{1}(\omega_{0})<\pi_{1}(\omega_{j})\leq\pi_{1}(\omega_{n})$ for
all $j=1,\dots,n$, where $\pi_{1}$ denotes projection onto the first
coordinate. A key observation for the proof of
\Cref{thm:CC-Exists-Intro} is that $\wt_{\kappa}$ is
supermultiplicative on bridges when $\kappa\geq 0$. This implies the
connective constant for bridges, $\mu_{\bridge}(\kappa)$, exists.

A classical argument due to Hammersley and Welsh shows that the number
of $n$-step self-avoiding bridges is the same, up to sub-exponential
corrections, as the number of $n$-step self-avoiding walks~\cite{HW};
see also~\cite[Section~3.1]{MadrasSlade}. An immediate consequence is
that $\mu_{\bridge}(0)=\mu(0)$.  To prove the existence of
$\mu(\kappa)$, we adapt the Hammersley-Welsh argument to $\kappa>0$;
i.e., we prove that the difference in the $\kappa$-ASAW weight of
$n$-step bridges and $n$-step walks is sub-exponential in $n$.

The Hammersley-Welsh argument involves ``unfolding'' self-avoiding
walks by reflecting segments of the walk through well-chosen
hyperplanes. It is during unfolding that the lack of
submultiplicativity must be overcome.

\subsection{Mean-field behaviour}
\label{sec:criticial-exponents}

To formulate \Cref{thm:High-D}, our main lace expansion result, we
require further assumptions on the step distribution $D$.
\begin{definition}
  \label{def:SD}
  Let $L>0$. A step distribution $D$ is \emph{spread-out with
      parameter~$L$} if it has the form
  \begin{equation}
    \label{eq:Steps}
    D(x) =
    \begin{cases}
      \frac{h(x/L)}{\sum_{x\in\Z^{d}\setminus \{o\}}h(x/L)} & x\neq o \\
      0 & x=o,
    \end{cases}
  \end{equation}
  where $h\colon \cb{-1,1}^{d}\to \co{0,\infty}$ is a piecewise continuous
  function such that $h(0)>0$, $0$ is a point of continuity, and
  \begin{enumerate}
  \item $h$ is invariant under the symmetries of $\Z^{d}$, and
  \item $\int h(x)\,dx =1$.
\end{enumerate}
\end{definition}

In what follows when we consider a ``spread-out step distribution'' we
mean the one-parameter family of step distributions obtained by
choosing a single function $h$. Note that $h(0)>0$ and $0$ being a
point of continuity for $h$ implies that the denominator in
\eqref{eq:Steps} is positive and $D(e_{1})=\slb>0$ if $L$ is taken
sufficiently large.  We will implicitly assume $L$ is at least this
large in what follows. The variance of $D$ will be denoted by
$\sigma^{2}\bydef\sum_{x\in\Z^{d}}\norm{x}_{2}^{2}D(x)$.

\begin{example}
  Consider $h(x)=2^{-d}$ on $\cb{-1,1}^{d}$. This leads to $D(x)$
  being uniformly distributed on vertices $x\neq o$ with
  $\norm{x}_{\infty}\leq L$.
\end{example}

The proof of \Cref{thm:CC-Exists-Intro} relies in part on establishing
that $\kappa$-ASAW is repulsive in an averaged sense; roughly
speaking, this means that a walk under the $\kappa$-ASAW law is not
typically attracted to its earlier trajectory. This idea, which is explained
more precisely in \Cref{sec:main-idea}, also turns out to enable a
lace expansion analysis of $\kappa$-ASAW at \emph{criticality}, a
notion which we introduce in the next two definitions.

\begin{definition}
  The \emph{susceptibility} of $\kappa$-ASAW is the power series
  \begin{equation}
    \label{eq:Susceptibility}
    \chi_{\kappa}(z) \bydef \sum_{n=0}^{\infty}c_{n}(\kappa)z^{n}.
  \end{equation}
\end{definition}

\begin{definition}
  \label{def:crit}
  The \emph{critical point} $z_{c}= z_{c}(D,\kappa)$ of $\kappa$-ASAW
  is defined to be
  \begin{equation*}
    z_{c} \bydef \sup\{z\geq 0 \mid \chi_{\kappa}(z)<\infty\}.
\end{equation*}
\end{definition}

For SAW, submultiplicativity implies that $z_{c}(0) =
\mu(0)^{-1}$. For $\kappa$ small enough that
\Cref{thm:CC-Exists-Intro} applies, it remains true that
$z_{c}(\kappa)=\mu(\kappa)^{-1}$ by the Cauchy-Hadamard
characterization of the radius of convergence.

Before stating our main result on the behaviour of $\kappa$-ASAW at
the critical point $z_{c}(\kappa)$, we require a few more
definitions. Precise formulations of the classes of walks involved in
these definitions can be found in \Cref{sec:conventions}.

The \emph{two-point function} of $\kappa$-ASAW is defined, for $z\geq
0$ and $x\in\Z^{d}$, by
\begin{equation}
  \label{eq:ASAW2PT}
  G_{z,\kappa}(x) \bydef 
  \sum_{n\geq 0}\sum_{\omega\in\saw_{n}(x)}z^{n}\wt_{\kappa}(\omega).
\end{equation}
Note that the inner sum is restricted to self-avoiding walks that end
at $x$; only $n$-step walks with positive probability under
$\PP_{n,\kappa}$ contribute.

The $\kappa$-ASAW two-point function should be compared with
the \emph{simple random walk} two-point function
\begin{equation}
  \label{eq:SRW2PT}
  \srwtwo_{z}(x) 
  \bydef
  \sum_{n\geq 0}\sum_{\omega\in\walk_{n}(x)}z^{n}\P_{n}\cb{\omega},
\end{equation}
in which the inner sum is over $\walk_{n}(x)$, the set of all $n$-step
walks with initial vertex $o\in\Z^{d}$ and terminal vertex
  $x\in \Z^{d}$. The term ``simple'' is used to
indicate that the associated law on $n$-step walks is $\P_{n}$,
although this may not be a nearest-neighbour walk.  Let $\mnorm{x}$
denote $\max\{\norm{x}_{2},1\}$, where $\norm{\cdot}_{2}$ is
  the Euclidean norm. Despite the notation, $\mnorm{\cdot}$ is not a
norm.
\begin{theorem}
  \label{thm:High-D}
  Let $d\geq 5$. For sufficiently spread-out step distributions, with
  parameter $L\geq L_{0}(D)$, there is a $\kappa_{0}>0$ such that if
  $0\leq \kappa\leq \kappa_{0}$ and $\alpha>0$ then
  \begin{equation}
    \label{eq:High-D-G}
    G_{z_{c},\kappa}(x) = \frac{a_{d}}{\sigma^{2}\mnorm{x}^{d-2}}
    \ob{1+O(L^{\alpha-2}) + O\ob{\frac{L^{2}}{\mnorm{x}^{2-\alpha}}}},
  \end{equation}
  where $\sigma^{2}$ is the variance of the step distribution $D$, the
  constants implicit in the $O(\cdot)$ notation may depend on $\kappa$
  and $\alpha$, and $a_{d} = 2^{-1} \pi^{-d/2} d
  \Gamma(\frac{d}{2}-1)$. Here $\Gamma(\frac{d}{2}-1)$ is the
  evaluation of Euler's Gamma function at $\frac{d}{2}-1$.
\end{theorem}
\Cref{thm:High-D} shows that the critical two-point function of
$\kappa$-ASAW has the same asymptotics as the critical ($z=1)$
two-point function of simple random walk in $d\geq 5$. In the language
of critical exponents, see~\cite[p.12]{Slade}, this says that
$\eta=0$, i.e., this is a verification that $\kappa$-ASAW has
mean-field behaviour. We have not attempted to optimize the relation
between $\kappa_{0}$ and $L$ in our proof, as our primary interest is
in the existence of $\kappa_{0}>0$ for finite $L$.

It is typically difficult to apply the lace expansion to models
containing attracting interactions, as these attractions make it
difficult to obtain what are known as diagrammatic bounds. We are able
to overcome this difficulty as the {\em on average repulsion} that
$\kappa$-ASAW satisfies is compatible with calculating such
bounds. This is discussed in more detail in \Cref{sec:main-idea}. Once
the diagrammatic bounds are obtained the remaining part of the lace
expansion analysis is well understood and can be adapted from existing
arguments~\cite{HvdHS}. We recall how this can be done in
\Cref{app:GA}. The proof of \Cref{thm:High-D} is carried out
in~\Cref{sec:Lace}.

\subsection{Main idea}
\label{sec:main-idea}

The proofs of \Cref{thm:CC-Exists-Intro,thm:High-D} are essentially
independent, but they share a common idea which we explain here. 

Let $\saw$ denote the set of all self-avoiding walks, not necessarily
starting at the origin $o$. Writing a walk $\omega$ as a concatenation
$\omega = \omega^{1}\circ\omega^{2}$ of two subwalks determines an
interaction conditional on $\omega^{1}$, i.e.,
\begin{equation}
  \label{eq:Conditional-Interaction}
  \indicatorthat{\omega\in\saw} e^{-H_{\kappa}(\omega)} =
  \ob{\indicatorthat{\omega^{1}\in\saw}e^{-H_{\kappa}(\omega^{1})}}
  \ob{\indicatorthat{\omega \in\saw} e^{-H_{\kappa}(\omega^{2};\,\omega^{1})}},
\end{equation}
where this formula defines $H_{\kappa}(\cdot\,;\omega^{1})$. Explicitly,
\begin{equation}
  \label{eq:Conditional-Interaction-Expl}
 \exp(- H_{\kappa}(\omega^{2};\omega^{1})) \bydef
  (1+\kappa)^{\abs{\adj(\omega^{2})}}(1+\kappa)^{\abs{\adj(\omega^{1},\,\omega^{2})}}, 
\end{equation}
where $\adj(\omega^{1},\,\omega^{2})$ is the set of pairs of adjacent
edges  $\{f_{1},f_{2}\}$ with $f_{i}\in\omega^{i}$, $i=1,2$.

For SAW, i.e., $\kappa=0$, the interaction is trivial:
$e^{-H_{0}(\omega)}=e^{-H_{0}(\omega^{2};\omega^{1})}=1$. Submultiplicativity
therefore follows from~\eqref{eq:Conditional-Interaction} and the
observation that
\begin{equation*}
  \indicatorthat{\omega\in\saw} =
  \indicatorthat{\omega^{1}\circ\omega^{2}\in\saw}\leq
  \indicatorthat{\omega^{2}\in\saw}.
\end{equation*}
For $\kappa>0$ it is not generally true that
$e^{-H_{\kappa}(\eta;\,\omega^{1})} \leq e^{-H_{\kappa}(\eta)}$. See
\Cref{fig:Walk-Example} and consider splitting the walk into the
indicated subwalks.

\Cref{eq:Conditional-Interaction} highlights a tension between
self-avoidance and self-attraction. Energetic rewards of $(1+\kappa)$
due to the conditional interaction only occur if the walk $\omega^{2}$
has edges adjacent to edges in $\omega^{1}$. Such an edge in
$\omega^{2}$ carries an entropic penalty, as the potential
configurations of $\omega^{2}$ are reduced. Thus there is both an
entropic benefit and an energetic penalty to dropping the conditional
interaction due to $\omega^{1}$ when
\eqref{eq:Conditional-Interaction} is summed over a suitable class of
walks.

More explicitly, if $\omega^{2}$ contains an edge $\{x_{1},y_{1}\}$
adjacent to an edge $\{x_{2},y_{2}\}$ of $\omega^{1}$, then typically
there is a self-avoiding modification of $\omega^{2}$ that traverses
$\{x_{2},y_{2}\}$ instead of $\{x_{1},y_{1}\}$. The modified walk will
be longer than the original, and will be assigned zero weight by the
conditional interaction. However, it has positive $\kappa$-ASAW
weight. The entropic gain of ignoring~$\omega^{1}$ can therefore be
estimated by considering the possible modifications to~$\omega^{2}$
and estimating the energetic cost of the modifications. The energetic
cost decreases as $\kappa$ decreases, and for $\kappa$ sufficiently
small we will show that the entropic benefit outweighs the energetic
penalty. This idea, which involves a weighted version of the 
multivalued map principle~\cite[Section~2.0.1]{Hammond1}, has been fruitful in obtaining upper bounds
on the number of self-avoiding polygons of given
length~\cite{Hammond1}.

\subsection{Discussion}
\label{sec:disc}

Ueltschi~\cite{Ueltschi} considered a model of SAWs with an attracting
reward for pairs of nearest-neighbour vertices under the assumption
that the step distribution $D(x)$ and attraction strength $\kappa$
satisfy
\begin{equation}
  \label{eq:Ueltschi-D}
  \inf_{\abs{x-y}=1,y\neq 0} \frac{D(y)}{D(x)} = \Delta > 0, \qquad
  (1+\kappa)^{2d} \leq 1+ \frac{\Delta^{2}}{2d(1+\kappa)^{2d-1}}.
\end{equation}
Note that the condition on $D(x)$ in~\eqref{eq:Ueltschi-D} implies the
step distribution has infinite range. Given~\eqref{eq:Ueltschi-D} it
can be shown that the entropic reward of ignoring $\omega^{1}$
outweighs the energetic cost. The fact that $D(x)$ has infinite range
and is ``smooth'' allows the use of a length preserving transformation
to prove the model is submultiplicative. Using this idea Ueltschi also
carries out a lace expansion analysis via the inductive approach
of~\cite{vdHS}; the length-preserving nature of the transformation is
important for the application of the inductive method. Ueltschi's
result was significant for being the first application of the lace
expansion to a self-attracting random walk. Self-attracting
interactions, which are also called non-repulsive, are typically
difficult to handle with lace expansion
methods~~\cite[Section~6.3]{Slade}.

The problem of analysing models of self-attracting self-avoiding walks
under weaker hypotheses on the step distributions was raised by den
Hollander~\cite[Chapter 4.8(5)]{dH}, and our work addresses
  this question when $d\geq 5$. As described in \Cref{sec:main-idea} the
main idea is to combine energy-entropy methods with classical techniques
for self-avoiding walk.

Beyond our main theorems, an important aspect of this work is that it
suggests that energy-entropy methods may be more generally useful in
the context of the lace expansion. In particular there is no need to
restrict to length-preserving transformations as in~\cite{Ueltschi}
(although length-preserving transformations do simplify technical
aspects due to~\cite{vdHS}). This is significant as energy-entropy
methods should be a fairly robust way to overcome a lack of repulsion
caused by weak attractions. Roughly speaking, the key step in
  such an argument is to first subdivide an object, and then to prove the
  gain in conformational freedom that arises when forgetting one part
  outweighs the loss of energetic attractions. Our proof implements
  this strategy for $\kappa$-ASAW, and it is plausible it could
  be implemented for other models, e.g., weakly self-attracting
  lattice trees in high dimensions via an adaptation of~\cite{HaraSladeTree,HvdHS}.

It is worth noting that energy-entropy arguments are carried out by
finding a transformation that estimates the number of new
configurations that are available. Finding a transformation is a
combinatorial and analytic problem, in contrast to other approaches to
overcoming a lack of repulsion via correlation
inequalities~\cite{Sakai,Slade}, resummation
identities~\cite{Helmuth}, or asymmetry
assumptions~\cite{vdHH}.

We end this section by mentioning two recent related works on
self-avoiding random walks subject to self-attraction. Firstly, there
has been interesting progress~\cite{BSW} on den Hollander's problem
for \emph{weakly} self-avoiding walk (WSAW) with a contact
self-attraction when $d=4$. The authors prove Gaussian decay of the
critical two-point function when the self-attraction and
self-repulsion strengths are sufficiently small by making use of a
rigorous renormalization group analysis. The techniques of~\cite{BSW}
are wholly different than those of the present paper, and an analysis
of self-attracting WSAW when $d\geq 5$ via lace expansion techniques
would be a very interesting complement to the results of~\cite{BSW}.
Secondly, in~\cite{PetrelisTorri} it has been shown that a related
model known as \emph{prudent self-avoiding walk} undergoes a collapse
transition in $d=2$ when the self-attraction is strong enough.

\section{Initial definitions, path transformations}
\label{sec:Transformations}

\subsection{Conventions}
\label{sec:conventions}

By a common abuse of notation $\Z^{d}$ will denote the $d$-dimensional
hypercubic lattice, i.e., the graph with vertex set $\Z^{d}$ and edge
set $E(\Z^{d})\bydef\{\{x,y\} \mid \norm{x-y}_{1}=1\}$. Recall
that the standard generators of $\Z^{d}$ will be denoted
$e_{1},\dots, e_{d}$. $\abs{A}$ will denote the cardinality of a
finite set $A$, and $A\sqcup B$ will denote the union of
  disjoint sets $A$ and $B$.

For $n\in\N \bydef\{0,1,2,\dots\}$, an \emph{$n$-step walk} is
a sequence $(\omega_{i})_{i=0}^{n}$, where $\omega_{i}\in\Z^{d}$ for
$0\leq i\leq n$ and $\omega_{i}\neq \omega_{i-1}$, $1\leq i\leq
n$. For such a walk let $\abs{\omega}\bydef n$, and for
  $0\leq i<j\leq \abs{\omega}$ we write $\omega_{\cb{i,j}}$ to denote
  the walk $(\omega_{i}, \omega_{i+1}, \dots, \omega_{j})$. Let
$\walk_{n}(x,y)$ be the set of $n$-step walks with $\omega_{0}=x$ and
$\omega_{n}=y$. We omit the first argument if $x=o$, the origin of
$\Z^{d}$, and let $\walk(x,y) \bydef \bigsqcup_{n\geq
  0}\walk_{n}(x,y)$. We also let $\walk_{n}$ denote the set of
  all $n$-step walks, with no constraints on the initial or final vertices.

A walk is \emph{self-avoiding} if $\omega_{i}\neq\omega_{j}$ for
$i\neq j$, and is a \emph{self-avoiding polygon} if $\abs{\omega}>2$,
$\omega_{i}\neq \omega_{j}$ for $0\leq i<j<\abs{\omega}$, and
$\omega_{0}=\omega_{\abs{\omega}}$. Note that polygons are
  rooted and oriented, which is a somewhat non-standard definition.
Let $\saw_{n}(x,y)$ denote the set of $n$-step self-avoiding walks
from $x$ to $y$ and $\sap_{n}(x)$ denote the set of $n$-step
self-avoiding polygons with initial vertex $x$. For self-avoiding
walks let $\saw_{n}(x) \bydef \saw_{n}(o,x)$, and again we will omit
the subscript $n$ to indicate a union over $n$.  $\saw$ and $\sap$
denote the sets of all self-avoiding walks and polygons, respectively,
with no restrictions on the initial vertex.

Let $\adj(A,B)$ denote the set of plaquettes spanned by pairs of
adjacent edges $e\in A$, $f\in B$ for subsets $A,B\subset E(\Z^{d})$,
and let $\adj(A) \bydef \adj(A,A)$. Let
$E(\omega) \bydef \{ \{\omega_{i},\omega_{i+1}\}\}_{i=0}^{\abs{\omega}-1}$
be the set of edges traversed by a walk $\omega$. By a slight abuse of
notation we will write $\adj(\omega,\eta)$ in place of
$\adj(E(\omega),E(\eta))$, and $\adj(\omega)\bydef \adj(\omega,\omega)$.

\subsection{Transformations by symmetries of $\Z^{d}$; basic path
  operations}
\label{sec:Symmetries}

For $x\in \Z^{d}$, let $\trans_{x}$ denote the operator of translation
by $x$, i.e., $\trans_{x}f(y) = f(y-x)$ for $f$ a function on
$\Z^{d}$. Translations will also act on subsets or collections of
subsets of $\Z^{d}$ by identifying sets with indicator functions. For
example, if $\omega\in \walk_{n}$, $\trans_{x}\,\omega$ is the $n$-step
walk $(\omega_{0}+x, \omega_{1}+x, \dots, \omega_{n}+x)$.

The projection operator $\pi_{i}\colon \Z^{d}\to \Z$ maps
$x = (x_{1}, \dots, x_{d})$ to $x_{i}$. To lighten notation, let
$\pi_{i}^{-1}(x) = \pi_{i}^{-1}(\pi_{i}(x))$ denote the hyperplane
passing through $x$ with normal $e_{i}$. The reflection operator
$\refl_{i}\colon \Z^{d}\to \Z^{d}$ reflects any vertex in the coordinate
hyperplane $\pi_{i}^{-1}(o)$.

If $\omega^{1}\in\walk_{m}$ and $\omega^{2}\in\walk_{n}$ their
\emph{concatenation} $\eta = \omega^{1}\circ \omega^{2}$ is the
$(n+m)$-step walk with $\eta_{i}=\omega^{1}_{i}$ for $0\leq i\leq m$,
and
$\eta_{m+i} = \trans_{(\omega^{1}_{m}-\omega^{2}_{0})}\,
\omega^{2}_{i}$ for $0\leq i\leq n$. This translation moves the
initial vertex of $\omega^{2}$ to the terminal vertex of $\omega^{1}$,
so the concatenated walk continues from where $\omega^{1}$ ends.

\subsection{Flips}
\label{sec:Flips}

Let $\omega$ be a walk and $P$ be a plaquette such that there is a
unique $i$ such that $\omega_{j}\in P$ iff $j\in \{i,i+1\}$.  We will
call such a plaquette $P$ \emph{flippable}. This
  condition implies that $\omega$ has exactly one edge in the
plaquette $P$, and no other vertices in $P$. Otherwise $P$ is
\emph{not flippable}. Since flippability is defined in terms of
$\omega$ we will write, for example, \emph{flippable for
  $\omega$} to indicate this dependence.

Suppose $P$ is a flippable plaquette for $\omega$, and that
$(\omega_{i},\omega_{i+1})$ is the unique edge of $\omega$ in $P$. The
\emph{flip of $\omega$ at $P$}, denoted $\flip_{P}(\omega)$, is the
walk $\omega^{\prime}$ that replaces $(\omega_{i},\omega_{i+1})$ with
the traversal of $P$ along the three edges distinct from
$\{\omega_{i},\omega_{i+1}\}$.  See \Cref{fig:Flip}. If $P$ is not
flippable, define $\flip_{P}(\omega) = \omega$. Two plaquettes
$P_{1}$ and $P_{2}$ are said to be \emph{disjoint} if they have no
vertices in common. Sets of disjoint flippable plaquettes are
what will be entropically important in what follows. The next three
lemmas establish useful properties of $\flip_{P}$.

\begin{lemma}
  \label{lem:Flipping-Disjoint}
  For any plaquette $P$ and vertices $x,y\in\Z^{d}$,
  $\flip_{P}\colon \walk(x,y)\to\walk(x,y)$, 
  $\flip_{P}\colon \saw(x,y)\to\saw(x,y)$, and
    $\flip_{P}$ is invertible. If $P'$ is disjoint from
  $P$, then $\flip_{P}\circ \flip_{P'}=\flip_{P'}\circ \flip_{P}$.
\end{lemma}
\begin{proof}
  To prove $\flip_{P}\colon \walk(x,y)\to\walk(x,y)$ it suffices to
  prove that $\flip_{P}$ does not change the endpoints of a walk. This
  is immediate as the first and last vertices of $\flip_{P}(\omega)$
  in $P$ are the same as the first and last vertices of $\omega$ in
  $P$.

  If $\omega\in\saw(x,y)$ and $P$ is not flippable for $\omega$,
  then $\flip_{P}(\omega)=\omega$, so the image is in $\saw(x,y)$. If
  $P$ is flippable, then $\omega$ contains one edge of $P$ and no
  other vertices; since $\flip_{P}$ only modifies $\omega$ on $P$ the
  result is self-avoiding.

  Invertibility of $\flip_{P}$ is clear, as if $P$ is not
  flippable for $\omega$ then $\flip_{P}$ is the identity, while
  if $P$ is flippable then $\omega$ can be recovered from
  $\flip_{P}(\omega)$ by replacing the traversal of three consecutive
  edges of $P$ in $\flip_{P}(\omega)$ by the traversal of the single
  unoccupied edge of $P$. Lastly, commutativity holds as flips at
  disjoint plaquettes modify the walk on disjoint sets of edges.
\end{proof}

Given a disjoint set of plaquettes $B=\{P_{1},\dots, P_{k}\}$, define
$\flip_{B}(\omega) \bydef
\flip_{P_{1}}\circ\dots\circ\flip_{P_{k}}(\omega)$; the commutativity
of flips at disjoint plaquettes implies the  definition of $\flip_{B}$ is unambiguous.
\begin{lemma}
  \label{lem:Flip-Invert}
  Let $\eta\circ\omega\in\saw$ and let $B\subset \adj(\eta,\omega)$ be
  a disjoint set of plaquettes that are flippable for $\omega$. Then
  $\omega$ is uniquely determined by $\flip_{B}(\omega)$ and $\eta$.
\end{lemma}
\begin{proof}
  Since the plaquettes in $B$ are disjoint, we can consider them
  separately. For each $P\in B$, there is at least one edge of $P$ in
  $\tilde\omega = \flip_{P}(\omega)$ that is also in $\eta$.  Suppose
  there is one edge $(\tilde\omega_{i},\tilde\omega_{i+1})$. Then $P$
  is the plaquette
  $\{\tilde\omega_{i-1}, \tilde\omega_{i}, \tilde\omega_{i+1},
  \tilde\omega_{i+2}\}$. If there are two edges of $\tilde\omega$ in
  $\eta$ then $P$ is the plaquette spanned by the two edges, and three
  or four edges being in $\eta$ contradicts $\eta\circ\omega\in\saw$.
  
  Thus, given $\eta$ and $\tilde\omega$, we can determine $B$. By
  \Cref{lem:Flipping-Disjoint} $\flip_{B}$ is invertible, so $\omega$
  is uniquely determined.
\end{proof}

Suppose a self-avoiding walk $\eta$ is composed of two subwalks
$\omega^{1}$ and $\omega^{2}$, i.e.,
$\eta = \omega^{1}\circ\omega^{2}$. It will be convenient to abuse
notation and write $\adj(\omega^{1},\omega^{2})$ in place of
$\adj(\omega^{1},\trans_{x}\omega^{2})$, where $\trans_{x}$ is the
translation that takes the initial vertex of $\omega^{2}$ to the final
vertex of $\omega^{1}$. As it will be contextually clear we are
discussing pairs of adjacent edges between $\omega^{1}$ and
$\omega^{2}$, this should not cause any confusion.

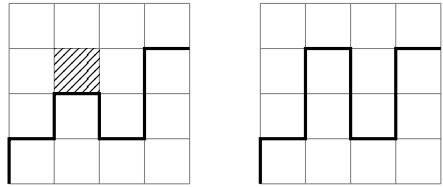
\begin{figure}[h]
  \centering
  \begin{tikzpicture}[scale=.6]
    \draw[gray] (0,0) grid (4,4);
    \draw[black, very thick] (0,0) -- (0,1) -- (1,1) -- (1,2) -- (2,2) -- (2,1) --
    (3,1) -- (3,3) -- (4,3);
    \fill[pattern= north east lines] (1,3) rectangle (2,2);
  \end{tikzpicture}
  \qquad
  \begin{tikzpicture}[scale=.6]
    \draw[gray] (0,0) grid (4,4);
    \draw[black, very thick] (0,0) -- (0,1) -- (1,1) -- (1,3) -- (2,3)
    --(2,1) --
    (3,1) -- (3,3) -- (4,3);
  \end{tikzpicture}
  \caption{Illustration of a flip applied to a self-avoiding walk
    $\omega$ at the shaded plaquette $P$, which is flippable.}
  \label{fig:Flip}
\end{figure} 

The next lemma says most plaquettes in $\adj(\omega^{1},\omega^{2})$
are flippable for $\omega^{2}$; to quantify this we define
\begin{equation}
  \label{eq:nrp}
  \nrp\bydef 2d(d-1).
\end{equation}

\begin{lemma}
  \label{lem:Adjacent-Flippable}
  If $\omega^{1}\circ\omega^{2}\in\saw$, there are at most $\nrp$
  plaquettes in $\adj(\omega^{1},\omega^{2})$ that are not
  flippable for $\omega^{2}$. If
  $\omega^{1}\circ\omega^{2}\in\sap$, there are at most $2\nrp$
  plaquettes in $\adj(\omega^{1},\omega^{2})$ that are not
  flippable for $\omega^{2}$.
\end{lemma}
\begin{proof}
  Without loss of generality, assume that $\omega^{2}_{0}$ is the
  endpoint of $\omega^{1}$, and call the vertices in common to
  $\omega^{2}$ and $\omega^{1}$ \emph{points of concatenation}. The
  proof characterises when $P\in\adj(\omega^{1},\omega^{2})$ is not
  flippable for $\omega^{2}$ case by case, depending on how many
  edges of $P$ are contained in $\omega^{1}\circ\omega^{2}$. By the
  definition of $P\in\adj(\omega^{1},\omega^{2})$ there are at
  least two such edges.
  
  First, note that a self-avoiding walk or polygon containing four
  edges in a single plaquette is a four step self-avoiding polygon. The
  claim is true in this case, as there are exactly two adjacent
  pairs of edges. Henceforth we may assume there are no
  plaquettes containing four edges.

  Suppose that $\omega^{1}$ and $\omega^{2}$ each contain exactly one
  edge of $P$. If $P$ does not contain a point of concatenation, then
  $P$ is flippable for $\omega^{2}$ since the two vertices
  of $P$ in $\omega^{1}$ are not in $\omega^{2}$. See
  \Cref{fig:Flippable-2e}. If $P$ contains a point of concatenation it
  may or may not be flippable for $\omega^{2}$.
 
  Suppose that $\omega^{1}\circ\omega^{2}$ contains three edges of
  $P$. Note that the three edges must occur sequentially in
  $\omega^{1}\circ \omega^{2}$ since this walk is self-avoiding or a
  self-avoiding polygon, and hence $P$ must contain a point of
  concatenation if it is to be in $\adj(\omega^{1},\omega^{2})$.  If
  two edges belong to $\omega^{1}$, then $P$ is flippable for
  $\omega^{2}$. Otherwise, $P$ is not flippable for
  $\omega^{2}$. See \Cref{fig:Flippable-3e}. 

  Thus $P\in\adj(\omega^{1},\omega^{2})$ and $P$ not being
    flippable implies there is a point of concatenation in $P$.
  As there are at most two points of concatenation, this verifies the
  claim, as each vertex of $\Z^{d}$ is contained in $\nrp$ plaquettes.
\end{proof}

\begin{figure}[h]
  \centering
  \begin{subfigure}{2in}
    \centering
    \begin{tikzpicture}[dot/.style={},scale=1.3] 
      \node[dot] (ll) at (0,0) {};
      \node[dot] (tl) at (0,1) {};
      \node[dot] (lr) at (1,0) {};
      \node[dot] (tr) at (1,1) {};
      \draw[black,thick, fill=black!25] (ll) circle (3pt);
      \draw[black,thick, fill=black!25] (lr) circle (3pt);
      \draw[black,thick, fill=black!25] (tl) circle (3pt);
      \draw[black,thick, fill=black!25] (tr) circle (3pt);
      \draw[black,very thick] (ll) -- (tl);
      \draw[black,dashed,very thick] (tr) -- (lr);
      \draw[gray,thin] (ll) -- (lr);
      \draw[gray,thin] (tl) -- (tr);
    \end{tikzpicture}
    \qquad
    \begin{tikzpicture}[dot/.style={},scale=1.3] 
      \node[dot] (ll) at (0,0) {};
      \node[dot] (tl) at (0,1) {};
      \node[dot] (lr) at (1,0) {};
      \node[dot] (tr) at (1,1) {};
      \draw[black,thick, fill=black!25] (ll) circle (3pt);
      \draw[black,thick, fill=black!25] (lr) circle (3pt);
      \draw[black,thick, fill=black!25] (tl) circle (3pt);
      \draw[black,thick, fill=black!25] (tr) circle (3pt);
      \draw[black,very thick] (ll) -- (lr);
      \draw[black,dashed,very thick] (lr) -- (tr);
      \draw[gray,thin] (ll) -- (tl);
      \draw[gray,thin] (tl) -- (tr);
    \end{tikzpicture}
    \caption{$P$ contains $2$ edges of
      $\omega^{1}\circ\omega^{2}$.}
    \label{fig:Flippable-2e}
  \end{subfigure}
  \quad
  \begin{subfigure}{2in}
    \centering
    \begin{tikzpicture}[dot/.style={},scale=1.3] 
      \node[dot] (ll) at (0,0) {};
      \node[dot] (tl) at (0,1) {};
      \node[dot] (lr) at (1,0) {};
      \node[dot] (tr) at (1,1) {};
      \draw[black,thick, fill=black!25] (ll) circle (3pt);
      \draw[black,thick, fill=black!25] (lr) circle (3pt);
      \draw[black,thick, fill=black!25] (tl) circle (3pt);
      \draw[black,thick, fill=black!25] (tr) circle (3pt);
      \draw[black,very thick] (ll) -- (lr) -- (tr);
      \draw[black,dashed,very thick] (tr) -- (tl);
      \draw[gray,thin] (ll) -- (tl);
    \end{tikzpicture}
    \qquad
    \begin{tikzpicture}[dot/.style={},scale=1.3] 
      \node[dot] (ll) at (0,0) {};
      \node[dot] (tl) at (0,1) {};
      \node[dot] (lr) at (1,0) {};
      \node[dot] (tr) at (1,1) {};
      \draw[black,thick, fill=black!25] (ll) circle (3pt);
      \draw[black,thick, fill=black!25] (lr) circle (3pt);
      \draw[black,thick, fill=black!25] (tl) circle (3pt);
      \draw[black,thick, fill=black!25] (tr) circle (3pt);
      \draw[black,very thick] (ll) -- (lr);
      \draw[black,dashed,very thick] (lr) -- (tr) -- (tl);
      \draw[gray,thin] (ll) -- (tl);
    \end{tikzpicture}
    \caption{$P$ contains $3$ edges of $\omega^{1}\circ\omega^{2}$. }
    \label{fig:Flippable-3e}
  \end{subfigure}
  \caption{Illustration of the cases arising in the proof of
    \Cref{lem:Adjacent-Flippable}. Solid black lines represent edges
    of $\omega^{1}$, while dashed black lines represent edges of
    $\omega^{2}$.}
  \label{fig:Flippable}
\end{figure}
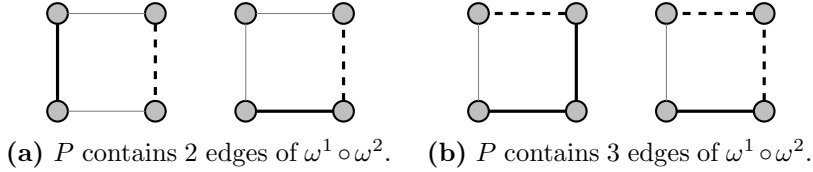

The cost
$\flipcost = \frac{\P_{n+2}(\omega^{\prime})}{\P_{n}(\omega)}$ is the
additional cost of the modified walk $\omega^{\prime}$ according to
the \emph{a priori} measure $\P$.  Recall the definition of
$\wt_{\kappa}$ in~\eqref{eq:Hamiltonian}. The next lemma is our basic
estimate for the energetic penalty of a flip. For future reference,
define
\begin{equation}
  \label{eq:tflipcost}
  \tflipcost \bydef \slb^{-2}(1+\kappa)^{2d-4}.
\end{equation}

\begin{lemma}
  \label{lem:TFC-Internal}
  Let $\omega$ be a self-avoiding walk and $P$ a flippable
  plaquette. Let $\omega^{\prime}=\flip_{P}(\omega)$. Then
  \begin{equation}
    \label{eq:TFC-Internal}
    \frac{\wt_{\kappa}(\omega^{\prime})}{\wt_{\kappa}(\omega)}
    \geq \tflipcost^{-1}. 
  \end{equation}
\end{lemma}
\begin{proof}
  The factor of $\flipcost$ comes from comparing the \emph{a priori}
  measures in the definition of $\wt_{\kappa}$. What remains is to
  bound the difference $\abs{\adj(\omega)} - \abs{\adj(\omega')}$. 

  The flip creates at least one pair of adjacent edges in $\omega'$
  that was not present in $\omega$, namely the pair of adjacent edges
  of $\omega'$ in $P$. There are $2d-2$ edges that are potentially
  adjacent to the unique edge of $\omega$ in $P$, and the hypothesis
  of $P$ being flippable for $\omega$ implies at least one of these
  edges is not in $\omega$. Thus at most $2d-3$ adjacent pairs of
  edges in $\omega$ are not present in $\omega'$. This
  proves~\eqref{eq:TFC-Internal}.
\end{proof}

In what follows it will be necessary to flip many
plaquettes. \Cref{lem:Flipping-Disjoint} guarantees the result will be
self-avoiding if the flipped plaquettes are disjoint. The next lemma
guarantees that every collection of plaquettes has a positive density
subset of disjoint plaquettes. Define $\alpha = \alpha(d)$ by
\begin{equation}
  \label{eq:alpha}
  \alpha(d) \bydef \frac{1}{1+8(d-1)^{2}}.
\end{equation}

\begin{lemma}
  \label{lem:P-Fraction}
  Given a finite set $A$ of plaquettes in $\Z^{d}$, there exists a
  subset of pairwise disjoint plaquettes of $A$ of size $\ceil{\alpha
    \abs{A}}$.
\end{lemma}
\begin{proof}
  The constant $\alpha$ is $(1+R)^{-1}$, where $R$ is the number of
  plaquettes $P'\neq P$ sharing a vertex with $P$. The remainder of
  the proof verifies the value of $R$ claimed in
    \eqref{eq:alpha} by performing inclusion-exclusion on the number
    of vertices $P'\neq P$ shares with $P$; note that this number is
    at most two.

  Any vertex $x$ is contained in exactly $4{d\choose 2}$
  plaquettes. To see this, note these plaquettes are in bijection with
  sets $\{(u_{i},\sigma_{i})\}_{i=1,2}$, where $u_{1}\neq u_{2}$ are
  distinct generators of $\Z^{d}$ and $\sigma_{i}\in\{\pm 1\}$: these
  sets identify the unique plaquette containing
  $x, x+ \sigma_{1}u_{1}, x+\sigma_{2}u_{2}$. Since every edge belongs
  to exactly $2(d-1)$ plaquettes, the total number of plaquettes
  sharing a vertex with a plaquette $P$ is therefore
  \begin{equation*}
    R = 4\ob{ (4{d\choose 2}-1) - (2(d-1)-1)} = 8(d-1)^{2},
  \end{equation*}
  where the factors of $-1$ correct for the presence of $P$ in our
  counts.
\end{proof}

\section{Existence of the connective constant: proof of \Cref{thm:CC-Exists-Intro}}
\label{sec:CC}

\subsection{Half-space walks and bridges}
\label{sec:Bridges}

We begin by recalling the basic definitions used in the
Hammersley-Welsh argument.

\begin{definition}
  The set $\bridge_{n}$ of \emph{$n$-step bridges} is the subset of
  $\omega\in \saw_{n}$ such that
  \begin{equation}
    \label{eq:Bridge}
    0=\pi_{1}(\omega_{0})< \pi_{1}(\omega_{i})\leq\pi_{1}(\omega_{n}),
    \qquad 1\leq i\leq n.
  \end{equation}
  The set $\hsw_{n}$ of \emph{$n$-step half-space walks} is the set of
  $\omega\in \saw_{n}$ such that
  \begin{equation}
    \label{eq:HS-Walk}
    0=\pi_{1}(\omega_{0})< \pi_{1}(\omega_{i}),\qquad 1\leq i\leq n.
  \end{equation}
\end{definition}
Let $\hsw\bydef\sqcup_{n\geq 0}\hsw_{n}$ and
$\bridge\bydef\sqcup_{n\geq 0}\bridge_{n}$.  The \emph{masses}
$h_{n}(\kappa)$ of $n$-step half-space walks and $b_{n}(\kappa)$ of
half-space bridges are defined by
\begin{equation}
  \label{eq:mass}
  h_{n}(\kappa) \bydef \sum_{\omega\in \hsw_{n}}\wt_{\kappa}(\omega),
  \qquad b_{n}(\kappa) \bydef \sum_{\omega\in\bridge_{n}}\wt_{\kappa}(\omega).
\end{equation}
Note that the inclusions $\bridge_{n}\subset\hsw_{n}\subset\saw_{n}$ imply that
$b_{n}(\kappa)\leq h_{n}(\kappa)\leq c_{n}(\kappa)$.

While self-attraction on adjacent edges ruins the
submultiplicativity of self-avoiding walks, it enhances the
supermultiplicativity of bridges.
\begin{proposition}
  \label{prop:Bridge-CC}
  Let $\kappa\geq 0$. The limit
  $\mu_{\bridge}(\kappa) = \lim_{n\to \infty}
  (b_{n}(\kappa))^{1/n}$ exists and is equal to
  $\sup_{n\geq 1} \big( b_{n}(\kappa) \big)^{1/n}$.
\end{proposition}
\begin{proof}
  The definition of a bridge implies that the concatenation
  $\omega^{1}\circ\omega^{2}$ of two bridges $\omega^{1}$ and
  $\omega^{2}$ is a bridge.
  Each pair of adjacent edges in $\omega^{i}$ remains adjacent in
  $\omega^{1}\circ\omega^{2}$. Any other pair of adjacent edges in the
  concatenation receives weight $1+\kappa\geq 1$, so
  \begin{equation*}
    \sum_{\omega^{1}\in\bridge_{n_{1}}}\sum_{\omega^{2}\in\bridge_{n_{2}}}
    \wt(\omega^{1})\wt(\omega^{2}) \leq
    \sum_{\eta\in \bridge_{n_{1}+n_{2}}}
    \wt(\eta)
    \indicatorthat{\eta=\omega^{1}\circ\omega^{2},\omega^{i}\in\bridge_{n_{i}}} 
  \end{equation*}
  for all choices of $n_{1},n_{2}\in\N$. 
  The left-hand side is $b_{n_{1}}(\kappa)b_{n_{2}}(\kappa)$. Ignoring
  the indicator on the right-hand side gives an upper bound
  $b_{n_{1}+n_{2}}(\kappa)$. Thus $b_{n}(\kappa)$ is
  supermultiplicative, and the proposition follows by Fekete's
  lemma. 
\end{proof}

\subsection{Unfolding I. Classical unfolding}
\label{sec:CU}

This section recalls how half-space walks can be unfolded into a
concatenation of bridges. We do this because a multivalued extension
of this procedure will be introduced in the next section. We omit the
proofs of the various facts that we recall; for details
see~\cite[Section~3.1]{MadrasSlade}.

\begin{definition}
  \label{def:Bridge-Point}
  Let $\omega\in\hsw$. The \emph{(first) bridge point}
  $\bp(\omega)$ of $\omega$ is the maximal index $i$ satisfying
  $\pi_{1}(\omega_{i}) = \max_{j}\pi_{1}(\omega_{j})$.
\end{definition}

\begin{definition}
  \label{def:span}
  The \emph{span} of a self-avoiding walk is
  \begin{equation}
    \label{eq:span}
    \spa(\omega) \bydef\max_{j}\pi_{1}(\omega_{j}) - \min_{j}\pi_{1}(\omega_{j}).
  \end{equation}
\end{definition}
Note that if $\omega$ is an $n$-step bridge then
$\spa(\omega)=\pi_{1}(\omega_{n})$.

Given a half-space walk $\omega\in\hsw_{n}$, let
$x\bydef-\omega_{\bp(\omega)}$ and define the \emph{initial bridge}
$\omega^{b} \bydef \omega_{\cb{0,\bp(\omega)}}$ and the
\emph{remainder}
$\omega^{h} \bydef\trans_{x}\,\omega_{\cb{\bp(\omega),n}}$.  We will
write $\omega=(\omega^{b},\omega^{h})$ in what follows to indicate
this decomposition into an initial bridge and a remainder. The
following properties of the decomposition are important.
\begin{enumerate}
\item $\omega = \omega^{b}\circ \omega^{h}$,
\item $\refl_{1}(\omega^{h})$ is a half-space walk,
\item $\spa(\omega^{b}) = \spa(\omega)$, and
\item $\spa(\refl_{1}(\omega^{h}))<\spa(\omega^{b})$, as $\omega$
  never revisits the coordinate hyperplane $\pi_{1}^{-1}(o)$ after
  $\omega_{0}$.
\end{enumerate}

\begin{definition}
  \label{def:CU}
  The \emph{classical unfolding map} $\Psi\colon \hsw\to \bridge$ is
  recursively defined as follows. $\Psi$ is the identity map on
  $\bridge$. Otherwise, if $\omega\in\hsw\setminus\bridge$, let
  $\omega = (\omega^{b},\omega^{h})$ and define
  $\Psi(\omega) = \omega^{b}\circ \Psi(\refl_{1}(\omega^{h}))$.
\end{definition}

In words, $\Psi$ reflects the remainder $\omega^{h}$ of the walk
$\omega$ through $\pi_{1}^{-1}(\omega_{\bp(\omega)})$, the affine
hyperplane with normal $e_{1}$ that contains the endpoint
$\omega_{\bp(\omega)}$ of~$\omega^{b}$. Since the reflection of
$\omega^{h}$ is itself a half-space walk, this procedure can be
iterated until the first bridge point of the newest half-space walk is
also the endpoint of the walk. The recursion terminates at some depth
$r=r(\omega)$ as the spans of the half-space walks produced are
strictly decreasing. See \Cref{fig:MVU-1,fig:MVU-2} for one step of
this procedure (the meaning of the shaded plaquettes will be explained
in the next section).

Thus, $\Psi$ produces a sequence of bridges $\omega^{b_{i}}$,
$i=1,\dots, r$, and
$\Psi(\omega)=\omega^{b_{1}}\circ \dots \circ \omega^{b_{r}}$. This
sequence of bridges is called the \emph{classical bridge
  decomposition} of $\omega$. In what follows, the bridges in the
classical bridge decomposition will always be denoted by
$\{\omega^{b_{i}}\}_{i=1}^{r}$. Let us record some properties of this
decomposition.
\begin{proposition}
  \label{prop:Classical-Unfolding}
    \leavevmode
  \begin{enumerate}
  \item $\spa(\Psi(\omega))=\sum_{i=1}^{r}\spa(\omega^{b_{i}})$, 
  \item the sequence $(\spa(\omega^{b_{i}}))_{i=1}^{r}$ of spans is
    strictly decreasing in $i$, and
  \item given $\Psi(\omega)$ and $\{\spa(\omega^{b_{i}})\}_{i=1}^{r}$,
    $\omega$ is uniquely determined.
  \end{enumerate}
\end{proposition}
Again we will not prove these claims, but let us remark that the third
property holds as knowing the lengths of the spans indicates the
locations at which to fold the bridge $\Psi(\omega)$ in order to undo
the unfolding procedure.

\subsection{Unfolding II. Multivalued unfolding}
\label{sec:MVU-II}

This section describes a multivalued extension of the classical
unfolding map. Roughly speaking, this multivalued extension quantifies
an entropic gain in unfolding a half-space walk $\omega$ into
bridges. The gain arises because the unfolded walk may have fewer
adjacent edges than the original walk $\omega$.

\begin{definition}
  \label{def:M-CU}
  The \emph{marked unfolding map $\bar\Phi$} is recursively defined on
  $\hsw$ by
  \begin{align}
    \label{eq:M-CU}
    \bar\Phi(\omega) 
    &\bydef 
     \ob{ (\omega^{b},\adj_{1}(\omega^{b},\omega^{h})),
      \bar\Phi(\refl_{1}(\omega^{h}))}, \\
    \label{eq:M-CU-2}
    \adj_{1}(\omega^{b},\omega^{h}) 
    &\bydef 
    \{P\in \adj(\omega^{b},\omega^{h}) \mid \textrm{$P$ flippable
      for $\omega^{b}$}\}.
  \end{align}
  Let
  $r\bydef r(\omega)$ denote the number of bridges generated by this
  recursion. The image of
  $\bar\Phi(\omega)$ will be denoted
  $((\omega^{b_{i}},\adj_{i}))_{i=1}^{r}$, where
  $\adj_{r}\bydef\emptyset$ and for $1\leq i<r$ we have used
    $\adj_{i}$ as shorthand for the set of plaquettes defined by
    \eqref{eq:M-CU-2} in the $i^{\text{th}}$ step of the recursion.
\end{definition}

The marked unfolding map is an extension of $\Psi$. It records the
plaquettes at which $\omega^{b_{i}}$ is flippable with respect to the
remainder of the half-space walk with initial bridge
$\omega^{b_{i}}$. Denote the set of discovered flippable plaquettes by
$\adj_{\bar\Phi}(\omega)$, i.e.,
\begin{equation*}
\adj_{\bar\Phi}(\omega) \bydef \bigsqcup_{i}\trans_{i}\,\adj_{i},
\end{equation*}
where $\trans_{i}$ is the translation that translates the bridge
$\omega^{b_{i}}$ to its location in the bridge $\Psi(\omega)$. 

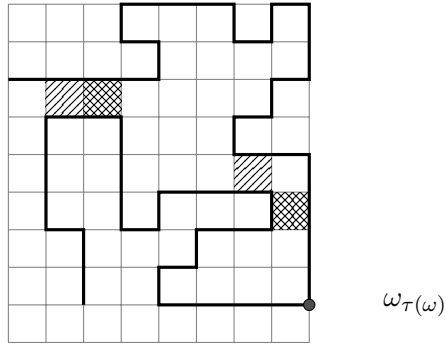
\begin{figure}
  \centering
  \begin{tikzpicture}[scale=.5]
    \phantom{$\omega_{\bp(\omega)}$}
    \draw[gray] (2,0) grid (10,9);
    \draw[black,very thick] 
    (2,7)
    -- (6,7) -- (6,8) -- (5,8) -- (5,9) -- (8,9) -- (8,8) -- (9,8) --
    (9,9) -- (10,9) -- (10,7) --(9,7) -- (9,6) -- (8,6) -- (8,5) --
    (10,5) -- (10,1);
    \draw[black,very thick] (10,1) -- (6,1) -- (6,2) -- (7,2) -- (7,3)
    -- (9,3) -- (9,4) -- (6,4) -- (6,3) -- (5,3) -- (5,6) -- (3,6) --
    (3,3) -- (4,3) -- (4,1);
    \fill[pattern = north east lines] (3,6) rectangle (5,7);
    \fill[pattern = north west lines] (4,6) rectangle (5,7);
    \fill[pattern = north east lines] (8,4) rectangle (9,5);
    \fill[pattern = north west lines] (9,3) rectangle (10,4);
    \fill[pattern = north east lines] (9,3) rectangle (10,4);
    \draw[black,fill = black!70] (10,1) circle (.15cm);
    \node at (10,1) [right] {$\omega_{\bp(\omega)}$};
  \end{tikzpicture}
  \caption{The figure depicts a half-space walk, along with its
    decomposition into $\omega^{b}$ and $\omega^{h}$. Shaded
    plaquettes are in $\adj(\omega^{b},\omega^{h})$; crosshatched
    plaquettes indicate a choice of subset of
    $\adj^{\star}(\omega^{b},\omega^{h})$.}
  \label{fig:MVU-1}
\end{figure}

For $k\in\N$, let $\hsw_{n}^{k}\subset \hsw_{n}$ be the set of
$n$-step half-space walks with $\abs{\adj_{\bar\Phi}(\omega)}=k$. To
define the multivalued extension of $\Psi$ on $\hsw_{n}^{k}$, we will
make use of the following facts. First, \Cref{lem:P-Fraction} implies
there exists a subset
$\adj^{\star}_{\bar\Phi}(\omega) \subset \adj_{\bar\Phi}(\omega)$ of
size $\ceil{\alpha k}$ such that the plaquettes in
$\adj^{\star}_{\bar\Phi}$ are pairwise disjoint. In what follows we
will assume the set $\adj^{\star}_{\bar\Phi}$ has been chosen
according to some arbitrary (but definite) procedure. Second,
\Cref{lem:Flipping-Disjoint} implies that if $B$ is a finite collection of
pairwise vertex-disjoint plaquettes and $x\in\Z^{d}$, then
$\flip_{B} = \prod_{P\in B}\flip_{P}$ is an unambiguously defined map
from $\saw(x)$ to $\saw(x)$.

\begin{definition}
  \label{def:MV-U}
  Let $0<\delta<\frac{1}{2}$, $n\in \N$, and $0\leq k\leq n$.  The
  \emph{multivalued unfolding map}
  $\Phi\colon \hsw^{k}_{n} \to 2^{\bridge_{n+2\ceil{\alpha \delta
        k}}}$ is defined by
  \begin{equation}
    \label{eq:MV-U}
    \Phi(\omega) \bydef \left\{\omega^{\prime} \, \Big| \, \omega^{\prime} =
    \flip_{B}(\Psi(\omega)),
    B\in {\adj^{\star}_{\bar\Phi}(\omega) \choose \ceil{\delta\alpha
        k}} \right\},
  \end{equation}
  where ${A\choose k}$ denotes the $k$-element subsets of a set
  $A$. 
\end{definition}

The next lemma gives the basic properties of $\Phi$; in particular it
verifies the stated codomain in the previous definition. To lighten
the notation, define
\begin{equation}
  \label{eq:ceils}
  \alpha_{k}\bydef\ceil{\alpha k}, \qquad \delta_{k}\bydef\ceil{\delta\alpha k}.
\end{equation}

\begin{lemma}
  \label{lem:Phi-Properties}
  Let $\omega\in\hsw^{k}_{n}$.
  Then 
  \begin{enumerate}
  \item Every $P\in \adj_{\bar\Phi}(\omega)$ is flippable for
    $\Psi(\omega)$.
  \item Each walk $\omega^{\prime}\in\Phi(\omega)$ is a bridge in
    $\bridge_{n+2\delta_{k}}$.
  \item The half-space walk $\omega$ can be reconstructed from any
    $\omega^{\prime}\in\Phi(\omega)$ given the spans
    $\spa(\omega^{b_{i}})$.
  \end{enumerate}
\end{lemma}
\begin{proof}
  Recall $\omega = \omega^{b_{1}}\circ\omega^{h}$. We begin by
    noting some properties of the plaquettes in
    $\adj(\omega^{b_{1}},\omega^{h})$. Since $\omega$ is a half-space
  walk, no plaquette in $\adj(\omega^{b_{1}},\omega^{h})$ contains
  vertices in the half-space
  $\pi_{1}^{-1}(\oc{-\infty,0})$. Similarly, no plaquette contains
  vertices in the half-space
  $\pi_{1}^{-1}(\co{\spa(\omega^{b_{1}})+1,\infty})$, as $\omega^{h}$
  is contained in $\pi_{1}^{-1}(\cb{1,\spa(\omega^{b_{1}})})$.

  We first prove (i). A moment of thought shows that each plaquette in
  $B\subset \adj_{1}(\omega^{b_{1}},\omega^{h})\subset
  \adj^{\star}_{\bar\Phi}(\omega)$ is flippable for
  $\omega^{b_{1}}\circ \refl_{1}(\omega^{h})$.  Iterating this
  argument for each bridge in the bridge decomposition of $\omega$
  implies each plaquette in $\adj_{\bar\Phi}(\omega)$ is flippable for
  $\Psi(\omega)$.

  We next prove (ii). The preceding shows each
  $\omega'\in\Phi(\omega)$ is given by
  $\omega'=\flip_{B}(\omega) =
  \flip_{B}(\omega^{b_{1}})\circ\dots\circ \flip_{B}(\omega^{b_{r}})$
  for $B\subset \adj_{\bar\Phi}^{\star}(\omega)$. This is because each
  plaquette in $B$ contains exactly one edge of $\Psi(\omega)$, and
  this edge is located in exactly one subwalk. By
  \Cref{lem:Flipping-Disjoint} and the first paragraph of the proof, each
  $\omega^{j} = \flip_{B}(\omega^{b_{j}})$ is a bridge with span
  $\spa(\omega^{b_{j}})$ for $j=1,\dots, r$. As a concatenation of
  bridges is a bridge and each flip adds exactly two edges, this
  completes the proof, as each $\omega'$ results from applying
  $\flip_{B}$ with $\abs{B}=\delta_{k}$.

  Lastly we prove (iii). By construction, the last bridge
  $\omega^{b_{r}}$ in the classical bridge decomposition has no flips
  applied to it in the formation of $\omega^{r}$. Given
  $\spa(\omega^{j})=\spa(\omega^{b_{j}})$ for each $j$,
  $\omega^{r}=\omega^{b_{r}}$ is determined. Hence, by
  \Cref{lem:Flip-Invert}, $\omega^{b_{r-1}}$ can be reconstructed.
  Iterating this procedure reconstructs $\omega$ from $\omega'$,
  establishing (iii).
\end{proof}

\begin{figure}
  \centering
    \begin{tikzpicture}[scale=.5]
    \draw[gray] (2,0) grid (18,9);
    \draw[black,very thick] 
    (2,7) -- (4,7) -- (4,6) -- (5,6) -- (5,7)
    -- (6,7) -- (6,8) -- (5,8) -- (5,9) -- (8,9) -- (8,8) -- (9,8) --
    (9,9) -- (10,9) -- (10,7) --(9,7) -- (9,6) -- (8,6) -- (8,5) --
    (10,5) -- (10,4) -- (9,4) -- (9,3) -- (10,3) -- (10,1);
    \draw[black,very thick] (10,1) -- (14,1) -- (14,2) -- (13,2) --
    (13,3) -- (11,3) -- (11,4) -- (14,4) -- (14,3) -- (15,3) -- (15,
    6) -- (17,6) -- (17,3) -- (16,3) -- (16,1);
    \fill[pattern = north east lines] (4,6) rectangle (5,7);
    \fill[pattern = north west lines] (4,6) rectangle (5,7);
    \fill[pattern = north west lines] (9,3) rectangle (10,4);
    \fill[pattern = north east lines] (9,3) rectangle (10,4);
    \draw[black,fill = black!70] (10,1) circle (.15cm);
    \node at (10,1) [below] {$\omega_{\bp(\omega)}$};
  \end{tikzpicture}
  \caption{The figure depicts the image in $\Phi(\omega)$ of the
    half-space walk $\omega$ depicted in \Cref{fig:MVU-2}, when the
    subset $B$ of plaquettes at which flips occur is the set of
    crosshatched plaquettes.}
  \label{fig:MVU-2}
\end{figure}

Having established the basic properties of the multivalued unfolding
map, we turn to estimating the weight of $n$-step half-space walks in
terms of the weights of bridges. The next lemma gives the basic
relation between these objects. 

For $n\in\N$, let $P(n)$ denote the number of partitions of $n$ into
distinct natural numbers. Let $\bridge_{n}(\ell)$ denote the set of
$n$-step bridges $\omega$ with $\spa(\omega)\leq \ell$, and
recall the definition of $\tflipcost$ from \eqref{eq:tflipcost}.
\begin{proposition}
  \label{lem:MUV-OS-kp}
  Consider SAW on $\Z^{d}$. For $n,k\in\N$ and $0\leq k\leq n$,
  \begin{equation}
    \label{eq:MUV-OS-kp}
    \sum_{\omega\in\hsw_{n}^{k}}\wt_{\kappa}(\omega)
    \leq
    P(n){\alpha_{k}\choose\delta_{k}}^{-1} \tflipcost^{\delta_{k}}
    \sum_{\omega'\in\bridge_{n+2\delta_{k}}(n)}(1+\kappa)^{k+3d(d-1)\sqrt{n}}
    \wt_{\kappa}(\omega') .
  \end{equation}
\end{proposition}
\begin{proof}
  We begin by comparing the weight of $\omega\in\hsw_{n}^{k}$ to the
  weight of a generic $\omega'\in\Phi(\omega)$. Recall that
  $\nrp=2d(d-1)$, and that $r=r(\omega)$ is the number of bridges
  created by the unfolding map.
  \begin{enumerate}
  \item At each unfolding, there are at most $\nrp$ plaquettes
      that are not flippable for $\omega^{b}$ in
    $\adj(\omega^{b},\omega^{h})$ by
    \Cref{lem:Adjacent-Flippable}. Hence,
    \begin{equation*}
      \wt_{\kappa}(\omega)\leq
      (1+\kappa)^{k+r\nrp}\wt_{\kappa}(\omega^{b_{1}}\circ \dots \circ
      \omega^{b_{r}}),
    \end{equation*}
    as there are $r$ unfolding steps.
  \item As
    $\omega'=\flip_{B}(\omega^{b_{1}}\circ\dots\circ\omega^{b_{r}})$
    and $\abs{B}=\delta_{k}$, applying \Cref{lem:TFC-Internal}
    $\delta_{k}$ times implies
    \begin{equation*}
      \wt_{\kappa}(\omega^{b_{1}}\circ\dots\circ\omega^{b_{r}})
      \leq
      \tflipcost^{\delta_{k}} 
      \wt_{\kappa}(\omega').
    \end{equation*}
  \end{enumerate}
  This implies
  \begin{equation}
    \label{eq:MUV-OS-k1}
    \wt_{\kappa}(\omega)\leq
    (1+\kappa)^{k+r(\omega)\nrp}\tflipcost^{\delta_{k}}\wt_{\kappa}(\omega').
  \end{equation}
  
  Next we apply the multivalued map principle. Let
    $\Phi(H_{n}^{k})=\cup_{\omega\in H_{n}^{k}}\Phi(\omega)$.
  \begin{align}
    \label{eq:MUV-OS-k2}
    \sum_{\omega\in\hsw_{n}^{k}}\wt_{\kappa}(\omega) \abs{\Phi(\omega)} 
    &=  \sum_{\omega\in\hsw_{n}^{k}} \sum_{\omega'\in\Phi(\omega)} \wt_{\kappa}(\omega)
    \\
    &=
    \sum_{\omega'\in\Phi(\hsw_{n}^{k})}
      \sum_{\omega\in\Phi^{-1}(\omega')} \wt_{\kappa}(\omega) \\
    &\leq \sum_{\omega'\in\Phi(\hsw_{n}^{k})}
      \sum_{\omega\in\Phi^{-1}(\omega')}
      (1+\kappa)^{k+r(\omega)\nrp}\tflipcost^{\delta_{k}}\wt_{\kappa}(\omega').
  \end{align}
  where the inequality is by~\eqref{eq:MUV-OS-k1}. To make the
  inner sum uniform in $\omega$, we use
  \Cref{prop:Classical-Unfolding}. The spans of the bridges in the
  bridge decomposition of $\omega$ are distinct positive integers
  summing to $\spa(\Psi(\omega))\leq n$. This implies that $r(\omega)$
  is at most $\frac{3}{2}\sqrt{n}$, as the sum of the first
  $\frac{3}{2}\sqrt{n}$ positive integers exceeds $n$. Hence
  \begin{equation}
    \label{eq:MUV-OS-k3}
    \sum_{\omega\in\hsw_{n}^{k}}\abs{\Phi(\omega)} \wt_{\kappa}(\omega)
    \leq 
    \sum_{\omega'\in\Phi(\hsw_{n}^{k})} \abs{\Phi^{-1}(\omega')}
    (1+\kappa)^{k+\frac{3}{2}\sqrt{n}\nrp}\tflipcost^{\delta_{k}}\wt_{\kappa}(\omega').
  \end{equation}
  Next we estimate $\abs{\Phi(\omega)}$ and
  $\abs{\Phi^{-1}(\omega')}$.
  \begin{enumerate}
    \item \Cref{lem:Phi-Properties} implies that
      \begin{equation}
        \label{eq:Phi-Card}
        \abs{\Phi(\omega)} = {\alpha_{k}\choose \delta_{k}}.
      \end{equation}
      Note that this is the number of subsets $B$ of
      $\adj_{\bar\Phi}^{\star}(\omega)$ of size $\delta_{k}$. The
      claim is true as (i) all plaquettes in $B$ are flippable for
      $\Psi(\omega)$, and (ii) every distinct choice of a subset $B$
      in the definition of $\Phi$ results in a distinct image.
    \item By \Cref{lem:Phi-Properties}~(iii) we can reconstruct
      $\omega$ from an image $\omega'\in\Phi(\omega)$ given the
      sequence of spans in the classical bridge decomposition of
      $\omega$. The maximal possible span of $\Psi(\omega)$, the
      bridge produced by the classical bridge decomposition, is $n$,
      and by \Cref{prop:Classical-Unfolding} the spans form a
      partition of $\spa(\Psi(\omega))$. Thus the number of preimages
      $\abs{\Phi^{-1}(\omega')}$ is at most $P(n)$, the number of
      partitions of $n$.
  \end{enumerate}
  This establishes \eqref{eq:MUV-OS-kp} if the index set
  $\bridge_{n+2\delta_{k}}(n)$ of the sum on the right-hand side is
  replaced with $\Phi(\hsw_{n}^{k})$. By
  \Cref{lem:Phi-Properties}~(ii),
  $\Phi(\hsw_{n}^{k})\subset \bridge_{n+2\delta_{k}}(n)$, as
  $\spa(\omega^{\prime}) = \spa(\Psi(\omega))\leq n$. The proposition
  thus follows as each summand is non-negative.
\end{proof}

\begin{lemma}
  \label{lem:Decay}
  For $\kappa$ sufficiently small, there are $\delta>0$,
  $K>0$, and $\con>0$ such that
  \begin{equation}
    \label{eq:MUV-Bound-1}
    {\alpha_{k}\choose \delta_{k}}^{-1}
    (\tflipcost \mu_{\bridge}^{2})^{\delta_{k}}
    (1+\kappa)^{k} \leq Ke^{-\con k}
  \end{equation}
  for all $k\in\N$.
\end{lemma}
\begin{proof}
  We will show that for $0<\delta<\frac{1}{2}$ small enough there is a
  $\kappa'$ such that for $\kappa\in\cb{0,\kappa'}$, there are $\con,K>0$
  such that~\eqref{eq:MUV-Bound-1} holds. This implies the statement
  of the lemma. Since it suffices to prove the validity
  of~\eqref{eq:MUV-Bound-1} for $k>k'\bydef(\delta\alpha)^{-1}$, we will
  restrict attention to such $k$.
  
  We begin by estimating the combinatorial prefactor
  in~\eqref{eq:MUV-Bound-1}. Recall the definition of $\alpha_{k}$ and
  $\delta_{k}$ in~\eqref{eq:ceils}. By using (i) 
  ${n\choose k}\geq (n/k)^{k}$ for $1\leq k\leq n$, (ii) $\max\{1,x\}
  \leq \ceil{x}\leq x+1$ for $x>0$, and (iii) $k>k'$ and
  $\delta<\frac{1}{2}$, we obtain
  \begin{equation}
    \label{eq:MUV-Bound-2}
    {\alpha_{k}\choose \delta_{k}}^{-1}  \leq
    \ob{\frac{\delta_{k}}{\ceil{\alpha k}}}^{\delta_{k}} 
       \leq (\delta + (\alpha k)^{-1})^{\delta_{k}}
          \leq (2\delta)^{\delta\alpha k}.
  \end{equation}
  Thus, when $k>k'$, the left-hand side of~\eqref{eq:MUV-Bound-1} is
  bounded above by
  \begin{equation}
    \label{eq:MUV-Bound-2.1}
    \left[2\delta(\tflipcost \mu^{2}_{\mathrm{\bridge}})^{\frac{\ceil{\delta\alpha
            k}}{\delta\alpha k}}
      (1+\kappa)^{(\delta\alpha)^{-1}}\right]^{\delta\alpha k},
  \end{equation}
  and the claim will follow  by showing that the quantity in square
  brackets is strictly less than one.

  By \Cref{prop:Bridge-CC},
  $\mu_{\bridge}(\kappa) = \sup_{n} \big( b_n (\kappa) \big)^{1/n}$,
  and this latter quantity is at most
  $\sup_{n} \big( c_{n}(\kappa) \big)^{1/n}$.  This, in turn, is
  bounded above by $(1+\kappa)^{2d-2}$, as each edge in a
  self-avoiding walk can be adjacent to at most $2d-2$ others. Since
  $\delta\alpha k>1$ when $k>k'$, we can bound above the bracketed
  quantity of~\eqref{eq:MUV-Bound-2.1} by
  \begin{equation}
    \label{eq:MUV-Bound-2.2}
   2\delta(\tflipcost (1+\kappa)^{4(d-1)})^{2}
    (1+\kappa)^{(\delta\alpha)^{-1}}.
  \end{equation}
  Recalling the definition~\eqref{eq:tflipcost} of $\tflipcost$, it follows that when
  $\kappa=0$ the quantity in~\eqref{eq:MUV-Bound-2.2}
  is strictly less than one, for $\delta=\delta(\slb)$ sufficiently
  small. Since the expression in~\eqref{eq:MUV-Bound-2.2} is
  continuous in $\kappa$ for $\delta$ fixed, it is strictly less than
  one for small positive $\kappa$. This proves the claim.
\end{proof}

The next theorem is a consequence of a stronger result due to Hardy
and Ramanujan; it will be needed to estimate the mass of $n$-step
half-space walks.
\begin{theorem}[Hardy--Ramanujan~\cite{HR}]
  \label{thm:Partition-Count}
  Let $P(n)$ denote the number of partitions of $n$ into distinct
  parts. Then
  \begin{equation}
    \label{eq:Partition-Count}
    \log P(n) \sim \pi\sqrt{\frac{n}{3}},\qquad n\to\infty,
  \end{equation}
  meaning that the ratio of the two sides tends to $1$ as $n\to\infty$.
\end{theorem}

\begin{proposition}
  \label{prop:MVU-Bound}
  Consider SAW on$\Z^{d}$. For $\kappa$ sufficiently small, there are
  $\con_{1},K_{1}>0$ such that
  \begin{equation}
    \label{eq:MVU-Bound}
    h_{n}(\kappa) = \sum_{\omega\in\hsw_{n}} \wt_{\kappa}(\omega)
    \leq  K_{1}e^{\con_{1}\sqrt{n}} \mu_{\bridge}^{n}.
  \end{equation}
\end{proposition}
\begin{proof}
  By \Cref{lem:MUV-OS-kp} and \Cref{lem:Decay},
  \begin{equation*}
    \label{eq:MVU-Bound-1}
    \sum_{\omega\in\hsw_{n}^{k}}\wt_{\kappa}(\omega) \leq
    Ke^{-a k} 
    \mu_{\bridge}^{-2\delta_{k}}
    (1+\kappa)^{3d(d-1)\sqrt{n}}
    P(n) \sum_{\omega'\in 
      \bridge_{n+2\delta_{k}}(n)}\wt_{\kappa}(\omega^{\prime}).
  \end{equation*}
  By \Cref{prop:Bridge-CC},
  $b_{\ell}(\kappa) \leq \mu_{\bridge}(\kappa)^{\ell}$. Hence, dropping
  the constraint that the spans of bridges on the right-hand side
  of~\eqref{eq:MVU-Bound-1} are at most $n$ yields
  \begin{equation}
    \label{eq:MVU-Bound-1.1}
    \sum_{\omega\in\hsw_{n}^{k}}\wt_{\kappa}(\omega) \leq
    Ke^{-a k}
    (1+\kappa)^{3d(d-1)\sqrt{n}}
    P(n) \mu_{\bridge}^{n}.
  \end{equation}
  The right-hand side of~\eqref{eq:MVU-Bound-1.1} is summable in $k$,
  and the left-hand side sums to $h_{n}(\kappa)$. By
  \Cref{thm:Partition-Count},
  $(1+\kappa)^{3d(d-1)\sqrt{n}} P(n)$ is at most
  $K_{1}e^{\con_{1}\sqrt{n}}$ for constants $K_{1},\con_{1}>0$; this
  completes the proof.
\end{proof}

\begin{proof}[Proof of \Cref{thm:CC-Exists-Intro}]
  To prove $\mu(\kappa) = \mu_{\bridge}(\kappa)$, it suffices to prove
  that there exist $K'$ and $\con'$ such that
  \begin{equation}
    \label{eq:DUB}
    b_{n}(\kappa) \leq c_{n}(\kappa)\leq
    K'e^{\con'\sqrt{n}}\mu_{\bridge}(\kappa)^{n}.
  \end{equation}

  For $\omega\in\saw_{n}$, let $m$ be the maximal $i$ such that
  $\pi_{1}(\omega_{i})$ is minimized. Let
  $\omega^{1} = \omega_{\cb{0,m}}$ and
  $\omega^{2}=\trans_{(-\omega_{m})}\,\omega_{\cb{m,n}}$. The first
  part of the proof is to define a multivalued map $\Psi$ that assigns
  to $(\omega^{1},\omega^{2})$ a set of pairs of half-space walks
  $\{(\eta^{1},\eta^{2})\}$ in an injective way. Note that
  $\omega^{2}$ is a half-space walk.

  The \emph{reversal} of a walk $\eta=(\eta_{i})_{i=1}^{m}$ is the
  walk $(\eta_{m-i})_{i=0}^{m-1}$; this walk begins at the endpoint of
  $\eta$ and goes back to the initial vertex. Translate the reversal
  of $\omega^{1}$ so that the walk begins at the origin, i.e.,
  consider the reversal of $\trans_{(-\omega_{m})}\omega^{1}$. This is
  almost a half-space walk; the only problem is that it may visit the
  coordinate hyperplane $\pi_{1}^{-1}(o)$ more than just at the
  initial vertex.

  We now define $\Psi$. Set $\eta^{2} \bydef \omega^{2}$. Define
  $\tilde\omega^{1}$ to be the concatenation of the one-step
  self-avoiding walk from $o$ to $e_{1}$ with the reversal of
  $\omega^{1}$.  Let $x \bydef e_{1}-\omega_{m}$ denote the vector
  along which $\omega^{1}$ is translated when forming this
  concatenation. Note that $\tilde\omega^{1}$ is a half-space walk.

  Let $\adj = \adj(\omega^{1},\omega^{2})$, and suppose this set
  contains $k$ flippable plaquettes for $\omega^{1}$. Let
  $\adj^{\star}\subset\adj$ be a subset of disjoint flippable
  plaquettes of size~$\alpha_{k}$; such a subset exists by
  \Cref{lem:P-Fraction}. Then
  \begin{equation*}
    \Psi(\omega) \bydef \left\{ (\eta^{1},\eta^{2}) \mid \eta^{1} =
      \flip_{\trans_{x}B}(\tilde\omega^{1}), B\in { \adj^{\star} \choose
        \delta_{k}}\right\}.
  \end{equation*}
  By an argument as in the proof of \Cref{lem:Phi-Properties},
  $\eta^{1}$ is a half-space walk, as any flips do not change the
  minimal value of the first coordinate of $\tilde{\omega}^{1}$. Note
  that if $\eta^{2}\in\bridge_{m}$, then
  $\eta^{1}\in\bridge_{n-m+2\delta_{k}+1}$.

  We now verify that it is possible to reconstruct
  $(\omega^{1},\omega^{2})$, and hence $\omega$ itself, from any image
  $(\eta^{1},\eta^{2})$.  Given $\eta^{1}$, the translation applied to
  $\omega^{1}$ is determined, as flips do not change the endpoints of
  walks. The translation applied to $\omega^{1}$ determines the
  translation applied to $\eta^{2}$, and hence determines
  $\omega^{2}$. Hence, by \Cref{lem:Flip-Invert}, $\omega^{1}$ is
  determined since we know $\omega^{2}$ and $\eta^{1}$.

  Recall that $\sqcup$ indicates a union of disjoint sets, and note
  $\saw_{n} = \sqcup_{k=0}^{n}\saw_{n}^{k}$, where $\saw_{n}^{k}$ is
  the set of self-avoiding walks such that
  $\abs{\adj(\omega^{1},\omega^{2})}=k$.  Proceeding as in the proof
  of \Cref{lem:MUV-OS-kp} yields
  \begin{equation*}
    c_{n}(\kappa)
    \leq
    \sum_{m=0}^{n}\sum_{k\leq m} {\alpha_{k}\choose \delta_{k}}^{-1}
    (\tflipcost)^{\delta_{k}} (1+\kappa)^{k+\nrp} h_{m}(\kappa)
    h_{n-m+2\delta_{k}+1}(\kappa).
  \end{equation*}
  The exponent of $(1+\kappa)$ is $k+\nrp$ as we have used
  \Cref{lem:Adjacent-Flippable} in the unfolding step that created the
  pair of half-space walks. Using \Cref{prop:MVU-Bound} to estimate
  the factors $h_{\ell}(\kappa)$ and \Cref{lem:Decay} to estimate the
  remaining terms yields
  \begin{equation*}
    c_{n}(\kappa)
    \leq K'\sum_{m=0}^{n}
    \sum_{k\leq m} e^{-\con k}
    e^{\con_{1}\sqrt{m}}e^{\con_{1}\sqrt{n-m+2\delta_{k}+1}} \mu_{\bridge}^{n},
  \end{equation*}
  where $K'$ is the product of the constant prefactors. The inequality
  $\sqrt{x} + \sqrt{y} \leq \sqrt{2x+2y}$ combined with the fact that
  $k$ is at most $n$ implies a further upper bound of the form
  \begin{equation*}
    c_{n}(\kappa)\leq K''(n+1)e^{\con'\sqrt{n}}\mu_{\bridge}^{n}
  \end{equation*}
  for some $K'',\con'>0$. As this establishes~\eqref{eq:DUB}, the
  proof is complete.
\end{proof}

\section{Averaged submultiplicativity and initial consequences}
\label{sec:ASM-Total}

The transformation used to estimate entropic gains in \Cref{sec:CC}
flipped a fixed fraction $\delta\alpha$ of flippable plaquettes,
and this led to $\slb$-dependent constants in estimates, where
  we recall $\slb$ depends on the step distribution $D$. This was
convenient as it gave us precise control over the length added to a
walk. For a lace expansion analysis the lack of uniformity in $\slb$
complicates matters, and this section defines a greedier transformation
that enables estimates uniform in $\slb$ when $\kappa = \kappa(\slb)$
is chosen correctly.  The price to pay is less control over the
increase in length of a walk.

\subsection{$\kappa$-ASAW with a memory}
\label{sec:Memory}

We first define memories, which will play the role of boundary
conditions for self-avoiding walks. Recall that $\sap(o)$ is the set
of self-avoiding polygons that begin at the origin.
\begin{definition}
  \label{def:Memory}
  A \emph{memory} is a walk $\eta\in
  \bigcup_{x}\saw(x,o) \cup\sap(o)$.
\end{definition} 
By a slight abuse of notation, let $\omega\cap\eta$ denote the set of
vertices in common to two walks $\omega$ and $\eta$. We will now
define the law of $\kappa$-ASAW conditional on having memory
$\eta$. Let $\sawno \bydef \cup_{x}\saw_{n}(x)$ denote the set of
$n$-step SAW with initial vertex $o$. For $n\in\N$, $\kappa\geq 0$,
and a memory $\eta$, define \emph{$n$-step $\kappa$-ASAW with memory
  $\eta$} to be the law $\PP^{\eta}_{n,\kappa}$ on $\sawno$ given by
\begin{equation}
  \label{eq:Model-BC}
  \PP_{n,\kappa}^{\eta}(\omega) \propto \wt_{\kappa}(\omega;\eta)
  \indicatorthat{\omega\in\sawno,\,\omega_{\cb{1,n}}\cap\eta
    = \emptyset}, 
\end{equation}
where
\begin{equation}
  \label{eq:Hamiltonian-BC}
  \wt_{\kappa}(\omega;\eta) \bydef 
  e^{-H_{\kappa}(\omega;\eta)}
  \P_{n}(\omega), 
  \quad e^{-H_{\kappa}(\omega;\eta)} =
  (1+\kappa)^{\abs{\adj(\omega)}}(1+\kappa)^{\abs{\adj(\eta,\omega)}},
\end{equation}
where we have recalled the definition of $H_{\kappa}(\omega;\eta)$
from \eqref{eq:Conditional-Interaction-Expl} for the convenience of
the reader.  Thus $\PP_{n,\kappa}^{\eta}$ is supported on
self-avoiding walks that intersect $\eta$ only at their initial vertex
$\omega_{0}=o$, and $\PP_{n,\kappa}^{\eta}$ gives a reward of
$(1+\kappa)$ for each pair of adjacent edges in $\omega$ and for each
time an edge of $\omega$ is adjacent to an edge of $\eta$.

\subsection{Averaged submultiplicativity}
\label{sec:ASM}

Let $\eta$ be a memory.  Define
\begin{equation*}
  \label{eq:saw-n-k-eta}
  \saw_{n,k}^{\eta}(x) \bydef \{\omega\in \saw_{n}(x) \mid
  \PP^{\eta}_{n,\kappa}(\omega)>0,\abs{\adj_{1}(\omega,\eta)}=k\},
\end{equation*}
where (by a slight abuse of the notation in \eqref{eq:M-CU-2})
$\adj_{1}(\omega,\eta)$ is the set of plaquettes that are flippable
for $\omega$ in $\adj(\omega,\eta)$.  The \emph{$(n,k)$ two-point
  function with memory~$\eta$} is
\begin{equation}
  \label{eq:2P-F-FM}
  c_{n,k}^{\eta}(x) \bydef \sum_{\omega\in\saw_{n,k}^{\eta}(x)}
  \wt_{\kappa}(\omega;\eta).
\end{equation}
The dependence of $c_{n,k}^{\eta}(x)$ on $\kappa$ is left
implicit. Let $c_{n}^{\eta} (x) \bydef \sum_{k\geq 0} c_{n,k}^{\eta}(x)$.

\begin{definition}
  \label{def:2PT-M}
  Let $z\geq 0$, $\kappa\geq 0$, and $x\in \Z^{d}$. The
  \emph{two-point function $G^{\eta}_{z,\kappa}(x)$ with memory
    $\eta$} is defined to be
  \begin{equation}
    \label{eq:2PT-M}
    G_{z,\kappa}^{\eta}(x) \bydef \sum_{n\geq 0}z^{n}c_{n}^{\eta}(x)
  \end{equation}
  If $\eta=\emptyset$ this is identically
  the $2$-point function of
    $\kappa$-ASAW as defined in \eqref{eq:ASAW2PT}. 
\end{definition}

Recall that the critical point $z_{c}(\kappa)$ was specified in
\Cref{def:crit}. The next proposition will imply that $G^{\eta}_{z,\kappa}$
is well-defined when $z<z_{c}(\kappa)$.
\begin{definition}
  \label{def:z0}
  Define $z_{0}(\kappa)\bydef  (1+\kappa)^{-2(d-1)}$.
\end{definition}
This choice of $z_{0}$ will be useful in
\Cref{sec:proof-gaussian-decay} for making comparisons with
  simple random walk. Recall that $\nrp$ was defined
  in~\eqref{eq:nrp}.

\begin{proposition}
  \label{prop:ASM}
  Suppose that $\kappa_{0}(d,\slb)$ is sufficiently small and that
  $z\geq z_{0}(\kappa)$.  Then, for all $\kappa< \kappa_{0}$,
  \begin{equation}
    \label{eq:ASM}
    G^{\eta}_{z,\kappa}(x) \leq (1+\kappa)^{\nrp}G_{z,\kappa}(x).
  \end{equation}
  In particular, $G^{\eta}_{z,\kappa}(x)$ is an absolutely convergent power series
  when $\abs{z}<z_{c}(\kappa)$. 
\end{proposition}
\begin{proof}
  We begin by defining a multivalued map $\Phi$ that will quantify the
  entropic gain of forgetting a memory.  Let
  $\omega\in \saw_{m,k}^{\eta}$. By \Cref{lem:P-Fraction}, there
  exists a set $\adj^{\star} \subset \adj(\eta,\omega)$ of
  $\alpha_{k}=\ceil{\alpha k}$ disjoint flippable plaquettes for
  $\omega$. Define $\Phi$ by
  \begin{equation}
    \label{eq:Phi-full}
    \Phi(\omega) = \{\omega^{\prime} \mid \omega^{\prime} =
    \flip_{B}(\omega), B\subset \adj^{\star}\}.
  \end{equation}
  The definition of $\Phi$ implicitly depends on $\eta$
  through $\adj^{\star}$, but we suppress this dependence from the
  notation as $\eta$ is fixed.

   To begin, we compare the weight of $\omega\in\saw_{m,k}^{\eta}(x)$
   under $\wt_{\kappa}(\cdot;\eta)$ to the weight $\wt_{\kappa}(\cdot)$
   of its image $\Phi(\omega)$. We claim that
   \begin{equation}
     \label{eq:c2}
     z^{m}\wt_{\kappa}(\omega;\eta) 
     \leq 
     (1+\kappa)^{\nrp} \ob{\frac{(1+\kappa)^{\frac{k}{\alpha_{k}}}}
       {1+z^{2}\slb^{2}(1+\kappa)^{-(2d-4)}}}^{\alpha_{k}}
     \sum_{\omega^{\prime}\in\Phi(\omega)}
     z^{\abs{\omega^{\prime}}}\wt_{\kappa}(\omega^{\prime})
   \end{equation}
   To see this, note that at each $P\in\adj^{\star}$ a flip may or may
   not occur. Hence the definition of $\Phi$ implies
   \begin{equation}
     \label{eq:c2.1}
     \sum_{\omega^{\prime}\in\Phi(\omega)}
     z^{\abs{\omega^{\prime}}} \wt_{\kappa}(\omega^{\prime})
     \geq 
     \prod_{j=1}^{\ceil{\alpha
         k}}(1+z^{2}\slb^{2}(1+\kappa)^{-(2d-4)})
     z^{m}\wt_{\kappa}(\omega).
   \end{equation}
   Formally, this bound arises by applying \Cref{lem:TFC-Internal}
   $\abs{B}$ times to $\omega^{\prime}=\flip_{B}(\omega)$, and then
   summing over all $B\subset\adj^{\star}$. As
   $\omega\in \saw_{m,k}^{\eta}(x)$, there is a
   $\sigma\in\{0,1,\dots,\nrp\}$ such that
   $\wt_{\kappa}(\omega) = (1+\kappa)^{-k-\sigma}
   \wt_{\kappa}(\omega;\eta)$. The possible values for $\sigma$
   follows from the proof of \Cref{lem:Adjacent-Flippable}, which
   shows there are at most $\nrp$ plaquettes that are not flippable in
   $\adj(\omega,\eta)$, because such plaquettes only occur at points
   of concatenation.  Inserting this formula for
   $\wt_{\kappa}(\omega)$ into~\eqref{eq:c2.1} and rearranging
   gives~\eqref{eq:c2}.

   Since $\kappa\geq 0$,
   \begin{equation}
     \label{eq:cty}
     \frac{(1+\kappa)^{\frac{k}{\alpha_{k}}}}
     {1+z^{2}\slb^{2}(1+\kappa)^{-(2d-4)}}
     \leq
     \frac{(1+\kappa)^{\frac{1}{\alpha}}}{1+z^{2}\slb^{2}(1+\kappa)^{-(2d-4)}}.
   \end{equation}
   The right-hand side of~\eqref{eq:cty} is at most
   $(1+\kappa)^{\frac{1}{\alpha}}(1+z_{0}^{2}\slb^{2}(1+\kappa)^{-(2d-4)})^{-1}$,
   which is strictly less than one when $\kappa=0$.  The right-hand
   side of~\eqref{eq:cty} is continuous in $\kappa$ and hence is
   strictly less than one for small positive $\kappa$. Thus, for
   $\kappa$ sufficiently small and $z\geq z_{0}$, \eqref{eq:c2}
   implies
   \begin{equation}
     \label{eq:ASM-1}
     z^{m}\wt_{\kappa}(\omega;\eta)\leq (1+\kappa)^{\nrp}
     \sum_{\omega^{\prime}\in\Phi(\omega)} z^{\abs{\omega^{\prime}}}
     \wt_{\kappa}(\omega^{\prime}). 
 \end{equation}

 Sum~\eqref{eq:ASM-1} over
 $\omega\in\saw^{\eta}_{m}(x) = \sqcup_{k\geq
   0}\saw^{\eta}_{m,k}(x)$. To conclude the proof what must be shown
 is that $\Phi(\omega)\subset \saw(x)$, and that if
 $\omega^{i}\in\saw^{\eta}_{m}(x)$, $i=1,2$, then
  \begin{equation}
    \label{eq:Image}
    \omega^{1}\neq\omega^{2}\implies
    \Phi(\omega^{1})\cap\Phi(\omega^{2})=\emptyset.
  \end{equation}
  The first claim follows from \Cref{lem:Flipping-Disjoint}. To prove
  the second claim, suppose $\gamma^{i}\in\Phi(\omega^{i})$,
  $i=1,2$. Then as $\gamma^{i}$ is the result of flipping $\omega^{i}$
  at a disjoint set $B$ of flippable plaquettes,
  \Cref{lem:Flip-Invert} implies $\omega^{i}$ is determined by
  $\gamma^{i}$ and $\eta$. Hence if $\gamma^{1}=\gamma^{2}$, then
  $\omega^{1}=\omega^{2}$.
\end{proof}

For our lace expansion analysis it will be necessary to have a
  slight generalization of \Cref{prop:ASM}. Let $\eta$ be a memory and define
\begin{equation}
  \label{eq:Model-BCp}
  \bar{\PP}_{n,\kappa}^{\eta}(\omega) \propto
  \wt_{\kappa}(\omega;\eta)
  \indicatorthat{\omega\in\sawno,\omega_{\cb{1,n-1}}\cap\eta
    = \emptyset}.
\end{equation}
Thus $\bar{\PP}_{n,\kappa}^{\eta}$ is a law on $n$-step self-avoiding walks
that do not intersect $\eta$ except for (i) at $\omega_{0}=o$ and (ii)
possibly at $\omega_{n}$. Let $\bar G^{\eta}_{z,\kappa}(x)$ denote the
two-point function for walks with law $\bar{\PP}_{n,\kappa}^{\eta}$.

\begin{corollary}
  \label{cor:ASM-Ep}
  \Cref{prop:ASM} holds for $\bar G^{\eta}_{z,\kappa}$ in place of
  $G^{\eta}_{z,\kappa}$ after changing $(1+\kappa)^{\nrp}$ to
  $(1+\kappa)^{2\nrp}$. 
\end{corollary}
\begin{proof}
  The proof for $\bar G^{\eta}_{z,\kappa}$ is, \emph{mutatis
    mutandis}, the proof of \Cref{prop:ASM}. The extra factor of
  $(1+\kappa)^{\nrp}$ arises as the proof of
  \Cref{lem:Adjacent-Flippable} allows for the existence of $2\nrp$
  plaquettes in $\adj(\eta,\omega)$ that are not flippable. This is
  because there may be $\nrp$ such plaquettes for each point of
  concatenation, of which there are at most $2$.
\end{proof}

\begin{corollary}
  \label{cor:ASM-E-1}
  Both \Cref{prop:ASM} and \Cref{cor:ASM-Ep} hold when the two-point
  functions are restricted to sums over walks of length at least $m$ for
  $m\in\N$.
\end{corollary}
\begin{proof}
  The map $\Phi$ used in the proof of \Cref{prop:ASM} only increases
  the length of a walk.
\end{proof}

\subsection{First applications of averaged submultiplicativity}
\label{sec:ASM-Cons-1}

It will be necessary to temporarily consider $\kappa$-ASAW in a finite
volume. Precisely, let $\Lambda = (\Z/L\Z)^{d}$ be a torus of side
length $L$. Define
\begin{equation*}
  \PP_{n,\kappa,\Lambda}(\omega)\propto
  \ob{\prod_{i=0}^{n}\indicatorthat{\omega_{i}\in\Lambda}} \PP_{n,\kappa}(\omega).
\end{equation*}
Let $\chi_{\Lambda}(z)$ be the susceptibility associated to
$\kappa$-ASAW on $\Lambda$, i.e.,
\begin{equation}
  \label{eq:Susceptbility-FV}
  \chi_{\Lambda,\kappa}(z) \bydef \sum_{n\geq 0} z^{n}
  \sum_{\omega\in \sawno}
  \wt_{\kappa}(\omega)\prod_{i=0}^{n}\indicatorthat{\omega_{i}\in\Lambda}.
\end{equation}
Note that $\chi_{\Lambda,\kappa}(z)$ is a polynomial in $z$ whenever
$\Lambda$ is finite. We will attach the subscript $\Lambda$ to
  other finite volume quantities in an analogous way.

\begin{lemma}
  \label{lem:DI-Chi}
  Let $\kappa_{0}(d,\slb)$ be as in \Cref{prop:ASM}. Fix
  $\kappa<\kappa_{0}(d,\slb)$ and $z\geq z_{0}$. On a finite torus
  $\Lambda$,
  \begin{equation}
    \label{eq:DI-Chi}
    -\frac{d}{dz}\ob{\chi_{\Lambda,\kappa}(z)}^{-1}\leq (1+\kappa)^{\nrp}z_{0}^{-1}.
  \end{equation}
\end{lemma}
\begin{proof}
  The proof we will give is the standard one for self-avoiding
  walk, with averaged submultiplicativity (\Cref{prop:ASM}) replacing
  submultiplicativity. While \Cref{prop:ASM} is for $\kappa$-ASAW on
  $\Z^{d}$, the proof applies immediately to $\Lambda$ as well: the
  only property of the graph that was used in the proof is that flips
  of flippable plaquettes are well-defined.
  
  As $\chi_{\Lambda,\kappa}(z)$ is a polynomial in $z$, we can compute
  \begin{align}
    \label{eq:DIQ-1}
    \frac{d}{dz}\cb{z\chi_{\Lambda,\kappa}(z)}    
    &= \sum_{y\in \Lambda}\sum_{\omega\in\saw(y)}(\abs{\omega}+1)z^{\abs{\omega}}
      \wt_{\kappa}(\omega)\\
    \label{eq:DIQ-2}
    &= \sum_{x,y\in \Lambda}\sum_{\eta\in\saw(x)}
      \sum_{\omega'\in\saw(x,y)}
      z^{\abs{\eta}}z^{\abs{\omega'}} \wt_{\kappa}(\eta\circ\omega')
      \indicatorthat{\eta\circ\omega'\in \saw} \\
    \label{eq:DIQ-3}
    &= \sum_{x,y \in \Lambda}\sum_{\eta\in\saw(x)}z^{\abs{\eta}}
     \wt_{\kappa}(\eta) G^{\eta}_{z,\kappa,\Lambda}(y-x).
  \end{align}
  \Cref{eq:DIQ-2} follows by using the fact that for $\omega\in\saw$,
  \begin{equation*}
    \abs{\omega}+1 = \sum_{x}\indicator\{\textrm{$\omega_{i}=x$ for
      some $i$}\},
  \end{equation*}
  and then splitting $\omega$ into $\eta=\omega_{\cb{0,i}}$ and
  $\omega'$.  The third equality follows from the definition of the
  conditional weight $\wt_{\kappa}(\cdot\,;\eta)$ and the translation
  invariance of $G^{\eta}_{z,\kappa,\Lambda}$ on a torus.

  By (the finite-volume version of) \Cref{prop:ASM}, the memory
  $\eta$ can be ignored to obtain an upper bound. Summing over $y$
  results in a factor $\chi_{\Lambda,\kappa}(z)$, as does the sum over
  $x$. The result is
  \begin{equation}
    \label{eq:DIQ}
    \frac{d}{dz}\cb{z\chi_{\Lambda,\kappa}(z)}  \leq
    (1+\kappa)^{\nrp}(\chi_{\Lambda,\kappa}(z))^{2}.
  \end{equation}
  Computing the derivative on the left-hand side, rearranging, and using
  $\chi_{\Lambda,\kappa}(z)\geq 0$ proves the lemma.
\end{proof}

The validity of \Cref{lem:DI-Chi} for all $z\geq z_{0}$ and all finite
$\Lambda$ implies the continuity of the phase transition for
$\kappa$-ASAW and a mean-field lower bound on the critical
exponent $\gamma$. \Cref{prop:Gamma-LB} below is a formal statement of these
facts. We include a proof for completeness, although the
implication given \Cref{lem:DI-Chi} is well-known~\cite{Aizenman}. We
need one preparatory lemma.

\begin{lemma}
  \label{lem:expdecay}
  For any $z<z_{c}$ there are constants $c_{1}(z),c_{2}(z)$ such that
  $G_{z,\kappa}(x)\leq c_{1}(z)\exp(-c_{2}(z)\norm{x}_{\infty})$.
\end{lemma}
\begin{proof}
  This follows from the fact that the summation defining
  $G_{z,\kappa}(x)$ contains only walks of length at least
  $C\norm{x}_{\infty}$ for some $C>0$, and hence is contained in the
  tail of the convergent sum $\chi_{\kappa}(z)$.  For more details,
  see~\cite[p.11-12]{Slade}.
\end{proof}

\begin{proposition}
  \label{prop:Gamma-LB}
  Let $\kappa_{0}(d,\slb)$ be as in \Cref{prop:ASM}. For
  $0<\kappa<\kappa_{0}(d,\slb)$ and $z_{0}\leq
  z<z_{c}$, 
  \begin{equation}
    \label{eq:Gamma_LB}
    \chi_{\kappa}(z) \geq \frac{(1+\kappa)^{-\nrp}z_{0}}{z_{c}-z},
  \end{equation}
  where $\chi_{\kappa}(z)$ is defined by~\eqref{eq:Susceptibility}.
\end{proposition}
\begin{proof}
  Note that $\chi_{\Lambda,\kappa}(z)\geq 1$ due to the contribution
  of the zero-step walk.  Fixing $\epsilon>0$ and
  integrating~\eqref{eq:DI-Chi} from $z\geq z_{0}$ to $z_{c}+\epsilon$
  implies
  \begin{equation}
    \label{eq:Chi-DE-1}
    \ob{\chi_{\Lambda,\kappa}(z)}^{-1} -
    \ob{\chi_{\Lambda,\kappa}(z_{c}+\epsilon)}^{-1} \leq
    (1+\kappa)^{\nrp}z_{0}^{-1}(z_{c}-z) + 
    \epsilon (1+\kappa)^{\nrp}z_{0}^{-1}.
  \end{equation}
  Let $\bar\chi_{\Lambda,\kappa}(z)$ be the contribution to the
  infinite volume susceptibility $\chi_{\kappa}(z)$ due to walks
  restricted to the finite box $\Lambda$ without periodic boundary
  conditions. $\bar\chi_{\Lambda,\kappa}(z)$ is monotone in $\Lambda$
  and converges to $\chi_{\kappa}(z)$. Moreover,
  $\chi_{\Lambda,\kappa}(z)$ dominates $\bar\chi_{\Lambda,\kappa}(z)$,
  and hence $\ob{\chi_{\Lambda,\kappa}(z_{c}+\epsilon)}^{-1}$ tends to zero
  as $\Lambda\uparrow\Z^{d}$ by the definition of $z_{c}$.

  To conclude it is enough to prove that
  $\ob{\chi_{\Lambda,\kappa}(z)}^{-1}$ converges to $\ob{\chi_{\kappa}(z)}^{-1}$
  as $\Lambda\uparrow\Z^{d}$, as we can then take
  $\Lambda\uparrow \Z^{d}$ in~\eqref{eq:Chi-DE-1}, followed by
  $\epsilon\to 0$. To see this we note that \Cref{prop:ASM} implies
  \begin{equation*}
    \abs{\bar\chi_{\Lambda,\kappa}(z)-\chi_{\Lambda,\kappa}(z)}\leq
    \sum_{x\in\partial\Lambda}G_{z,\kappa}(x)\chi_{\Lambda,\kappa}(z)
  \end{equation*}
  by reasoning as from \eqref{eq:DIQ-1}--\eqref{eq:DIQ-3}; here
  $\partial\Lambda$ denotes the boundary vertices of
  $\Lambda$. Dividing through by $\chi_{\Lambda,\kappa}(z)$ and using
  that $G_{z,\kappa}(x)$ decays exponentially in $\norm{x}_{\infty}$
  by \Cref{lem:expdecay} proves that $\bar\chi_{\Lambda,\kappa}(z)$ has
  the same limit as $\chi_{\Lambda,\kappa}(z)$. This proves the claim
  as we have already noted that $\bar\chi_{\Lambda,\kappa}(z)$
  converges to $\chi_{\kappa}(z)$.
\end{proof}

\section{A lace expansion for $\kappa$-ASAW}
\label{sec:Lace}

This section derives and analyzes a lace expansion for $\kappa$-ASAW
to prove \Cref{thm:High-D}.  As self-avoiding walk is the special case
$\kappa=0$ of $\kappa$-ASAW, many aspects of our analysis mirror the
SAW case. In what follows we therefore focus on the details specific
to $\kappa\neq 0$, and give precise statements and references for the
details that we omit.

For the reader unfamiliar with the lace expansion, the following
pointers to the literature may be helpful. Our derivation of a lace
expansion for $\kappa$-ASAW in \Cref{sec:Lace-exp} is an adaptation of
the Brydges-Spencer expansion~\cite{BrydgesSpencer}; \cite[Ch.\
3.2--3.3]{Slade} and \cite[Ch.\ 5.2]{MadrasSlade} contain pedagogical
accounts of this method. For the derivation of diagrammatic bounds our
approach in \Cref{sec:Lace-bounds} is similar to \cite{HvdHS}, but the
unfamiliar reader may wish to consult~\cite[Ch.\ 4]{Slade}
or~\cite[Ch.\ 5.4]{MadrasSlade} as a first introduction to the notion
of diagrammatic bounds. It is at this step that we make use of
averaged submultiplicativity. Lastly, for the convergence of the
expansion and proof of \Cref{thm:High-D} we make use of the method of
\cite{HvdHS}. We recall the method of \cite{HvdHS} in \Cref{app:GA},
and our proof of \Cref{thm:High-D} is by a verification of the
hypotheses of this method.

\subsection{Explicit $\kappa$-ASAW interaction}
\label{sec:ingredients}

In this section we give an algebraic formulation of the
$\kappa$-ASAW weight. Let
\begin{equation}
  \label{eq:U-Interaction}
  U_{ij}(\omega) \bydef \indicatorthat{\omega_{i}=\omega_{j}} -
  \kappa\indicatorthat{
    \textrm{$\{\omega_{i},\omega_{i+1},\omega_{j-1},\omega_{j}\}$ is
      a plaquette}}.
\end{equation}

\begin{lemma}
  \label{lem:A-Rep}
  If $\omega\in\walk_{n}$ is an $n$-step walk with $\omega_{0}=o$, then
  $\PP_{n,\kappa}(\omega)$ is proportional to 
\begin{equation}
  \label{eq:Model-Explicit}
  \wt_{\kappa}(\omega) \indicatorthat{\omega\in\saw} = 
  \P_{n}(\omega)
  \mathop{\prod_{0\leq i<j\leq n}}_{j>i+1}(1-U_{ij}(\omega)).
\end{equation}
\end{lemma}
\begin{proof}
  Recall that $\omega_{i+1}\neq\omega_{i}$ by the definition of a
  walk.  As $\omega_{i}=\omega_{j}$ precludes
  $\{\omega_{i},\omega_{i+1},\omega_{j-1},\omega_{j}\}$ being a
  plaquette, the first indicator in~\eqref{eq:U-Interaction} encodes
  the self-avoidance constraint between $\omega_{i}$ and $\omega_{j}$
  for $j>i+1$. The second indicator encodes the self-attraction
  between edges; each attraction is counted with weight $1+\kappa$
  when $\omega_{i}$ is the first visit to a plaquette and $\omega_{j}$
  is the last visit to the plaquette, which requires $j\geq i+3$.
\end{proof}

\subsection{Derivation of the lace expansion}
\label{sec:Lace-exp}

In this section we derive a lace expansion for $\kappa$-ASAW using the
so-called algebraic method~\cite{Slade}. A similar expansion for a
vertex attraction was derived in~\cite{Ueltschi}.

For $a\leq b$ integers let
$\cb{a,b} \bydef \{a,a+1,\dots, b\}$, and
  $\cb{a,a-1}=\emptyset$. An \emph{edge} $\{i,j\}$ is an element of
${\cb{a,b}\choose 2}$ with $\abs{i-j}>1$; $\{i,j\}$ will be
abbreviated $ij$. A \emph{graph} $G$ on $\cb{a,b}$ is a
(possibly empty) set of edges, and $G$ is \emph{connected} if
for all $j\in\cb{a+1,b-1}$ there are $i<j<k$ such that $ik$ is an edge
in $G$. Let $\graph_{\cb{a,b}}$ denote the set of graphs on
$\cb{a,b}$, and $\graphc_{\cb{a,b}}$ the set of connected graphs. Let
$\omega$ be a walk and define
\begin{align}
  \label{eq:K}
  K_{\cb{a,b}}(\omega) &\bydef \sum_{G\in\graph_{\cb{a,b}}} \prod_{ij\in
                         G}-U_{ij}(\omega), \quad\textrm{and} \\
  \label{eq:J}
  J_{\cb{a,b}}(\omega) &\bydef \sum_{G\in\graphc_{\cb{a,b}}} \prod_{ij\in G}-U_{ij}(\omega).
\end{align}
These definitions imply
$K_{\cb{a,a}}=K_{\cb{a,a+1}}=J_{\cb{a,a}}=J_{\cb{a,a+1}}=1$ due to the
contribution of the empty graph.

Expanding the product of $(1-U_{ij})$ over $ij$
in~\eqref{eq:Model-Explicit} and using~\eqref{eq:ASAW2PT} implies the
two-point function can be written in terms of graphs:
\begin{equation}
  \label{eq:K-G}
  G_{z,\kappa}(x) = \sum_{n=0}^{\infty}\sum_{\omega\in\walk_{n}(x)} z^{n}\P_{n}(\omega)
  K_{\cb{0,n}}(\omega).
\end{equation}
In what follows, we manipulate~\eqref{eq:K-G} by rewriting the term
$K_{\cb{0,n}}(\omega)$ to obtain a convolution equation for
$G_{z,\kappa}$. The first step is a well-known lemma.  For notational
convenience we will use the convention that $\sum_{j=b+k}^{b}f(j)=0$
if $k\geq 1$. 
\begin{lemma}[Lemma~5.2.2 of~\cite{MadrasSlade}]
  \label{lem:Graph-Recursion}
  For any walk $\omega$ and $a<b$,
  \begin{equation}
    \label{eq:Graph-Recursion}
    K_{\cb{a,b}}(\omega) = K_{\cb{a+1,b}}(\omega) +
    \sum_{j=a+2}^{b}J_{\cb{a,j}}(\omega)K_{\cb{j,b}}(\omega).
  \end{equation}
\end{lemma}

For $x\in\Z^{d}$, define $\Pi_{z,\kappa}(x)$ by
\begin{equation}
  \label{eq:Pi-pre}
  \Pi_{z,\kappa}(x)\bydef\sum_{n=2}^{\infty} \sum_{\omega\in\walk_{n}(x)}
  z^{n}\P_{n}(\omega) J_{\cb{0,n}}(\omega).
\end{equation}
It is not \emph{a priori} clear that the series defining
$\Pi_{z,\kappa}$ is convergent. When convergence is unknown we will
interpret the series and the formulas in which it occurs as formulas
relating formal power series in $z$. We recall that for
$f,g\colon\Z^{d}\to\R$ the convolution of $f$ and $g$ is
$(f\ast g)(x) \bydef \sum_{y\in\Z^{d}}f(y)g(x-y)$, and we also recall
that the step distribution $D$ was introduced in~\eqref{eq:Steps}.
\begin{proposition}[Theorem~5.2.3 of~\cite{MadrasSlade}]
  \label{prop:Recursion}
  As formal power series in $z$,
  \begin{equation}
    \label{eq:Recursion}
    G_{z,\kappa}(x) =  \indicatorthat{x=o} + (zD\ast G_{z,\kappa})(x) +
    (\Pi_{z,\kappa}\ast G_{z,\kappa})(x).
  \end{equation}
  This is an equality between functions when all terms of
    \eqref{eq:Recursion} are absolutely convergent in $z$. 
\end{proposition}
\begin{proof}
  Consider the contribution of $n$-step walks to
  $G_{z,\kappa}(\omega)$, i.e.,
 \begin{equation}
   \label{eq:e2}
    \sum_{\omega\in\saw_{n}(x)}z^{n}\wt_{\kappa}(\omega) =
    \sum_{\omega\in\walk_{n}(x)}z^{n}\P_{n}(\omega)K_{\cb{0,n}}(\omega).
  \end{equation}
  Since $U_{ij}(\omega)$ only depends on those $\omega_{\ell}$ with
  $i\leq\ell\leq j$, $K_{\cb{0,j}}(\omega)J_{\cb{j,n}}(\omega)$ is
  equal to
  $ K_{\cb{0,j}}(\omega_{\cb{0,j}})J_{\cb{j,n}}(\omega_{\cb{j,n}})$.
  Hence, by using~\eqref{eq:Graph-Recursion} to rewrite the factor
  $K_{\cb{0,n}}$ in~\eqref{eq:e2}, the right-hand side
  of~\eqref{eq:e2} can be rewritten as
  \begin{align}
    \label{eq:e3}
      &\hspace{-5mm} \sum_{\omega\in\walk_{n}(x)}
    zD(\omega_{1}-\omega_{0})z^{n-1}\P_{n-1}(\omega_{\cb{1,n}})
       K_{\cb{1,n}}(\omega) 
    \\\nonumber &+
      \sum_{\omega\in\walk_{n}(x)} 
\sum_{j=2}^{n}z^{j}\P_{j}\ob{\omega_{\cb{0,j}}}J_{\cb{0,j}}(\omega_{\cb{0,j}}) 
                  z^{n-j}\P_{n-j}\ob{\omega_{\cb{j,n}}}K_{\cb{j,n}}(\omega_{\cb{j,n}}).
  \end{align}

  To conclude, sum~\eqref{eq:e2} over all $n$. The left-hand side is
  $G_{z,\kappa}(x)$. For~$n=0$, the right-hand side is
  $\indicatorthat{x=o}$, the contribution of the zero-step walk. For
  $n\geq 1$ we use \eqref{eq:e3}. The
  factors of $G_{z,\kappa}$ arise by~\eqref{eq:K-G} and the
  translation invariance of $\wt_{\kappa}$. The term $\Pi_{z,\kappa}$
  arises from the sum over $j$ by~\eqref{eq:Pi-pre}.
  This proves~\eqref{eq:Recursion} in the sense of formal power series,
  as the coefficient of $z^{n}$ consists of only finitely many terms
  for all $n$.
\end{proof}

\subsection{Representation of $\pi^{(m)}_{z,\kappa}(x)$}
\label{sec:Pi-Refined}

To make use of \Cref{prop:Recursion} requires control on $\Pi_{z,\kappa}$. 
Obtaining this control is at the heart of the lace expansion, and
requires some further definitions. 

\begin{definition}
  \label{def:Lace}
  A connected graph $G$ is a \emph{lace} if for any edge $ij\in G$
  the graph $G\setminus ij$ is not connected. Let
  $\lace_{\cb{a,b}}^{(m)}$ denote the set of laces on $\cb{a,b}$ with
  $m$ edges, and
  $\lace_{\cb{a,b}} = \bigsqcup_{m\geq 1}\lace_{\cb{a,b}}^{(m)}$.
\end{definition}

\begin{lemma}[p.126-p.128 of~\cite{MadrasSlade}]
  \label{lem:Lace-Algorithm}
  For $G\in\graphc_{\cb{a,b}}$, let $\lacemap(G)$ be the graph
  with edges $\{s_{i}t_{i}\}$ defined by $s_{1}=a$,
  $t_{1} = \max \{t \mid s_{1}t\in G\}$, and
  \begin{equation*}
    \label{eq:Lace-Algorithm}
    t_{i+1} = \max\{ t \mid \textrm{$\exists$ $s<t_{i}$ such that
      $st\in G$}\}, \quad s_{i+1} = \min\{s \mid st_{i+1}\in G\}.
  \end{equation*}
  Then (i) $\lacemap(G)$ is a lace and (ii) if
  $\lacemap(G)=\lacemap(H)$ then $\lacemap(G\cup
  H)=\lacemap(G)$.
\end{lemma}
\Cref{fig:lacemap,fig:Diagram-Lace} 
depict a graph $G$ and its associated lace $\lacemap(G)$.
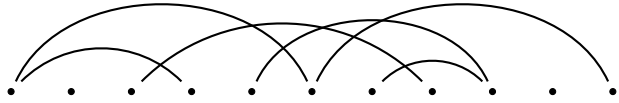
\begin{figure}[h]
  \centering
    \begin{tikzpicture}[scale=.8]
    \foreach \x in {0,1,...,10}
    \draw[fill=black!100] (\x,0) circle (.5mm) node (n\x) {} ;
    \draw[black,thick] (n0) edge [bend left=65] (n5);
    \draw[black,thick] (n4) edge [bend left=65] (n8);
    \draw[black,thick] (n5) edge [bend left=65] (n10);

    \draw[black,thick] (n0) edge [bend left=45] (n3);
    \draw[black,thick] (n2) edge [bend left=45] (n7);
    \draw[black,thick] (n6) edge [bend left=45] (n8);
  \end{tikzpicture}
  \caption{An illustration of the graph $G$ with edge set
    $\{ \{0,3\}, \{0,5\}, \{2,7\}, \{4,8\}, \{5,10\},
    \{6,8\}\}$. The vertices are labelled $0,1,\dots, 10$ from left
    to right.}
  \label{fig:lacemap}
\end{figure}

$\graph_{\cb{a,b}}$ is partially ordered by inclusion, and
\Cref{lem:Lace-Algorithm} implies that for every lace $L$ there is a
maximal graph $G$ such that $\lacemap(G)=L$. The edges of the maximal
graph $G$ can be partitioned into $L\sqcup \comp(L)$, where $\comp(L)$
are called the \emph{compatible edges}.  Explicitly, $ij\in \comp(L)$
if and only if $ij\notin L$ and $\lacemap(L\cup \{ij\})=L$. It follows
that
\begin{equation}
  \label{eq:Connected-Lace}
  J_{\cb{a,b}}(\omega) = \sum_{L\in\lace_{\cb{a,b}}}\prod_{ij\in
    L}-U_{ij}(\omega) \prod_{i'j'\in\comp(L)}(1-U_{i'j'}(\omega)).
\end{equation}

Define $J^{(m)}_{\cb{a,b}}$ to be the contribution
to~\eqref{eq:Connected-Lace} given by $L\in\lace_{\cb{a,b}}^{(m)}$,
and let
\begin{equation}
  \label{eq:Pi-n}
  \pi^{(m)}_{z,\kappa}(x)  
  \bydef \sum_{n=2}^{\infty} \sum_{\omega\in\walk_{n}(x)} z^{n}\P_{n}(\omega)
  J^{(m)}_{\cb{0,n}}(\omega). 
\end{equation}
These definitions imply that as formal power series
\begin{equation}
  \label{eq:Pi}
  \Pi_{z,\kappa}(x) = \sum_{m\geq 1}\pi^{(m)}_{z,\kappa}(x),
\end{equation}
and our approach to controlling $\Pi_{z,\kappa}$ will be to estimate
the terms $\pi^{(m)}_{z,\kappa}$.  This will be done by re-expressing
$\pi^{(m)}_{z,\kappa}$ in terms of sums of collections of walks
subject to conditional interactions $\wt_{\kappa}(\cdot\,;\eta)$ and
estimating these sums.  The remainder of this section introduces the
definitions needed to make the reformulation of $\pi^{(m)}_{z,\kappa}$
precise. As noted previously, the arguments are similar to standard
ones~\cite{Slade,MadrasSlade}; see also~\cite{Ueltschi}.

\begin{figure}[h]
  \centering
  \begin{tikzpicture}[scale=.8]
    \foreach \x in {0,1,...,10}
    \draw[fill=black!100] (\x,0) circle (.5mm) node (n\x) {} ;
    \draw[black,thick] (n0) edge [bend left=65] (n5);
    \draw[black,thick] (n4) edge [bend left=65] (n8);
    \draw[black,thick] (n5) edge [bend left=65] (n10);
    \draw[black,thick,decorate,decoration={brace,mirror,raise=4pt}]
    ([xshift=4pt]n0.west) --node[below=3mm]{$n_{1}$}
    ([xshift=-5pt]n4.east);
    \draw[black,thick,decorate,decoration={brace,mirror,raise=4pt}]
    ([xshift=4pt]n4.west) --node[below=3mm]{$n_{2}$} ([xshift=-5pt]n5.east);
    \draw[black,thick,decorate,decoration={brace,mirror,raise=4pt}]
    ([xshift=4pt]n5.west) --node[below=3mm]{$n_{4}$} ([xshift=-5pt]n8.east);
    \draw[black,thick,decorate,decoration={brace,mirror,raise=4pt}]
    ([xshift=4pt]n8.west) --node[below=3mm]{$n_{5}$} ([xshift=-5pt]n10.east);
  \end{tikzpicture}
  \caption{The lace $\lacemap(G)$ associated to the graph $G$ from
    \Cref{fig:lacemap} has the edge set
    $\{ \{0,5\},\{4,8\},\{5,10\}\}$. $\lacemap(G)$ subdivides the
    vertex set into intervals of length $n_{1}=4$, $n_{2}=1$,
    $n_{3}=0$, $n_{4}=3$, and $n_{5}=2$.}
  \label{fig:Diagram-Lace}
\end{figure}
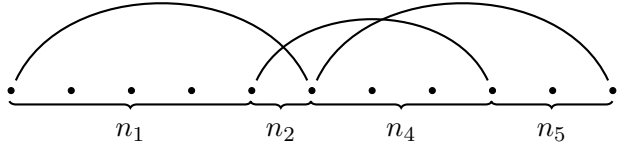

A lace $L\in \lace^{(m)}_{\cb{0,n}}$ partitions $\cb{0,n}$ into $2m-1$
intervals with disjoint interiors, and this induces a composition of
$n$ into $2m-1$ parts. See \Cref{fig:Diagram-Lace}. This composition
yields a vector $\V{n} \bydef (n_{1},n_{2},\dots, n_{2m-1})$ of
interval lengths with the properties
\begin{equation}
  \label{eq:Intervals}
  \begin{aligned}
    n_{j}&\geq 0, \qquad j\in \{3,5,7,\dots, 2m-3\}\\
    n_{j}& \geq 1, \qquad j\in \cb{0,2m-1}\setminus\{3,5,7,\dots,2m-3\},
\end{aligned}
\end{equation}
and $n=\textstyle{\sum_{j=1}^{2m-1}}n_{j}$. Conversely, to each such
$\V{n}$ there is associated a unique lace graph~\cite[Exercise
3.6]{Slade}.

Let $M_{i} \bydef \sum_{j=1}^{i}n_{j}$. By our convention for sums,
$M_{k}\bydef 0$ for $k\leq 0$. Define
$I_{j} \bydef \oc{M_{j-1},M_{j}}$ for $j=1,2,3,\dots,2m-1$ and
$I_{k}\bydef\emptyset$ for $k\leq 0$. Let
$\comp_{i}(L) \bydef \{i'j'\mid j'\in I_{i}\}$ be the set of
compatible edges $i'j'$ whose right endpoint $j'$ is in
$I_{i}$. Observing that $\comp(L)=\bigsqcup_{i=1}^{2m-1}\comp_{i}(L)$
since the intervals $I_{i}$ are a partition of $\oc{0,n}$, these
definitions imply that, for any lace $L\in\lace^{(m)}_{\cb{0,n}}$ and
walk $\omega\in\walk_{n}$,
\begin{equation}
  \label{eq:U-Edge}
  \prod_{ij\in\comp(L)}(1-U_{ij}(\omega)) 
  =
  \prod_{i=1}^{2m-1}\prod_{i'j'\in\comp_{i}(L)}(1-U_{i'j'}(\omega)).
\end{equation}

Because $\comp_{i}(L)$ consists of edges whose right endpoint is in
$I_{i}$, the second product on the right-hand side of
\eqref{eq:U-Edge} describes an interaction between the vertices
$\{\omega_{j}\}_{j\in I_{i}}$ and the vertices
$\{\omega_{j'}\}_{j'\leq M_{i}}$.  \Cref{prop:Lace-Edge} below
controls this interaction; achieving this control requires
characterising the edges $ij\in\comp(L)$.

Recalling that $M_{k}=0$ and $I_{k}=\emptyset$ for $k\leq 0$, the
compatible edges can be characterized as follows. Let
  $M^{k}_{2k-2}=M_{2k-2}$ if $k\neq 1$, and $M^{1}_{0}=\emptyset$. For $k\geq 0$,
$i<j$, and $m\geq 1$:
\begin{enumerate}
\item if $j\in I_{2k+1}$, then
  $i\in \bigcup_{\ell=2k-1}^{2k+1}I_{\ell}\cup\{M_{2k-2}\}$, and
  $j=M_{2m-1}=n$ implies $i>M_{2m-3}$.
\item if $j\in I_{2k+2}$, then
  $i\in \bigcup_{\ell=2k-1}^{2k+2}I_{\ell}\cup\{M_{2k-2}\}$, and
  $j=M_{2k+2}$ implies $i>M_{2k-1}$.
\end{enumerate}
This classification follows by considering the procedure defined in
\Cref{lem:Lace-Algorithm}. For $m=1$ the only incompatible edge is
$0n$. For $m\geq 2$ the constraints fall into three types: for
$I_{2k+2}$, $k\geq 1$; for $I_{2k+1}$, $k\leq m-2$; and for
$I_{2},I_{3},I_{4}$ and $I_{2m-1}$.
The case $I_{2}$ is distinguished because the only endpoint of a lace
between $s_{1}$ and $t_{1}$ is $s_{2}$, while for $i\geq 2$ we have
$s_{i}<t_{i-1}\leq s_{i+1}<t_{i}$. This distinction is manifested
above by the triviality of the intervals $I_{k}$ for $k\leq
0$. $I_{3}$ and $I_{4}$ are distinguished as the initial vertex of a
compatible edge cannot be $0$, and $I_{2m-1}$ is distinguished for a
similar reason. See \Cref{fig:Diagram-Lace-2}.

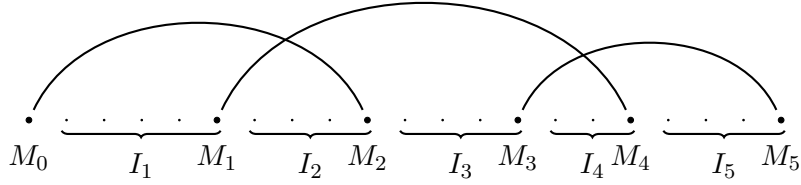
\begin{figure}[h]
  \centering
  \begin{tikzpicture}[scale=.5]
    \foreach \x in {0,1,...,20}
    \draw[fill=black!100] (\x,0) circle (.2mm) node (n\x) {} ;
    \draw[fill=black!100] (n0) circle (.75mm) {};
    \draw[fill=black!100] (n5) circle (.75mm) {};
    \draw[fill=black!100] (n9) circle (.75mm) {};
    \draw[fill=black!100] (n13) circle (.75mm) {};
    \draw[fill=black!100] (n16) circle (.75mm) {};
    \draw[fill=black!100] (n20) circle (.75mm) {};
    \draw[black,thick] (n0) edge [bend left=65] (n9);
    \draw[black,thick] (n5) edge [bend left=65] (n16);
    \draw[black,thick] (n13) edge [bend left=65] (n20);
    \draw[black,thick,decorate,decoration={brace,mirror,raise=4pt}]
    ([xshift=4pt]n1.west) --node[below=3mm]{$I_{1}$}
    ([xshift=-5pt]n5.east);
    \draw[black,thick,decorate,decoration={brace,mirror,raise=4pt}]
    ([xshift=4pt]n6.west) --node[below=3mm]{$I_{2}$} ([xshift=-5pt]n9.east);
    \draw[black,thick,decorate,decoration={brace,mirror,raise=4pt}]
    ([xshift=4pt]n10.west) --node[below=3mm]{$I_{3}$} ([xshift=-5pt]n13.east);
    \draw[black,thick,decorate,decoration={brace,mirror,raise=4pt}]
    ([xshift=4pt]n14.west) --node[below=3mm]{$I_{4}$}
    ([xshift=-5pt]n16.east);
    \draw[black,thick,decorate,decoration={brace,mirror,raise=4pt}]
    ([xshift=4pt]n17.west) --node[below=3mm]{$I_{5}$}
    ([xshift=-5pt]n20.east);
    \node[below=2mm] at (n0) {$M_{0}$};
    \node[below=2mm] at (n5) {$M_{1}$};
    \node[below=2mm] at (n9) {$M_{2}$};
    \node[below=2mm] at (n13) {$M_{3}$};
    \node[below=2mm] at (n16) {$M_{4}$};
    \node[below=2mm] at (n20) {$M_{5}$};
  \end{tikzpicture}
  \caption{A lace $\{ s_{1}t_{1},s_{2}t_{2},s_{3}t_{3}\}$ with
    $s_{1}=M_{0}$, $s_{2}=M_{1}$, and $s_{3}=M_{3}$. Its corresponding
    intervals $I_{i}$, $i=1,\dots, 5$ are also illustrated.  }
  \label{fig:Diagram-Lace-2}
\end{figure}

\begin{definition}
  \label{def:subwalks}
  For $\V{n}$ a composition of $n$ into $2m-1$ parts and
  $\omega\in\walk_{n}$, let
  $\V{\omega}\bydef (\omega^{(i)})_{i=1}^{2m-1}$ denote the vector of
  walks determined by
  $\omega^{(i)} \bydef \omega_{\cb{M_{i-1},M_{i}}}$. Note
  $\omega^{(i)}$ has length $n_{i}$.
\end{definition}
Using the characterisation of compatible edges above we will now
rewrite the right-hand side of \eqref{eq:U-Edge} in terms of the
subwalks $\omega^{(i)}$; this requires some further definitions. To
keep the notation to a minimum we will in fact give an upper bound.

When $m=1$ define
\begin{equation*}
  A^{(1)}_{1} \bydef \{\omega^{(1)}\in\saw\cup\sap\}, \quad \eta^{(1)}\bydef\emptyset.
\end{equation*}
For $m\geq 2$ we introduce: $\eta^{(1)}\bydef\emptyset$;
$\eta^{(2)}\bydef\omega^{(1)}$; for $k=1,\dots,m-1$,
$\eta^{(2k+1)}\bydef\omega^{(2k-1)}\circ\omega^{(2k)}$; and for
$k=2,\dots,m-1$,
$\eta^{(2k)}\bydef\omega^{(2k-3)}\circ\omega^{(2k-2)}\circ\omega^{(2k-1)}$. For
$m\geq 2$, $1\leq k\leq 2m-1$, and $L\in\lace^{(m)}$ define
$A^{(m)}_{k} = A^{(m)}_{k}(L)$ by
\begin{equation*}
  A^{(m)}_{k}\bydef
  \{\omega^{(k)}\in\saw, \textrm{$\omega^{(k)}_{i}\neq
    \eta^{(k)}_{j}$ if $i'j'\in \comp_{k}(L)$}\},
\end{equation*}
where $i'$ is the index such that $\omega^{(k)}_{i}$ is the $i'$th
vertex in $\omega = \omega^{(1)}\circ\dots\circ\omega^{(2m-1)}$, and
similarly for $j'$.

\begin{proposition}
  \label{prop:Lace-Edge}
  Let $m\geq 1$ and $1\leq k\leq 2m-1$ be integers, let
  $L\in\lace^{(m)}$ have an associated length vector $\V{n}$, let
  $\omega\in\walk_{n}$, and let $\omega^{(k)}$ be the $k^{\text{th}}$
  subwalk of $\omega$ determined by $L$ as in
  \Cref{def:subwalks}. Then
  \begin{equation}
    \label{eq:Lace-Edge}
    \P_{n_{k}}(\omega^{(k)})\prod_{i'j'\in\comp_{k}(L)}(1-U_{i'j'}(\omega))
    \leq \indicator_{A^{(m)}_{k}}
    \wt_{\kappa}(\omega^{(k)};\eta^{(k)}).
  \end{equation}
\end{proposition}
\begin{proof}
  The proposition is largely a translation of the
    definition of a compatible edge $i'j'\in\comp_{k}(L)$ and the
    definitions of the walks $\eta^{(k)}$. 
    
    Let us first consider $m=1$, so $k=1$. In this case
    $\eta^{(1)}=\emptyset$. Each factor $1-U_{i'j'}(\omega)$ for
    $i'j'\in\comp_{1}(L)$ enforces
    $\omega^{(1)}_{i'}\neq\omega^{(1)}_{j'}$ (in the case $k=1$ there
    is no distinction between primed and unprimed indices). Since the
    only edge that is \emph{not} compatible is $0n$, we see that
    $\omega_{\cb{0,n-1}}$ is self-avoiding, but
    $\omega_{n}=\omega_{0}$ is not forbidden. Hence
    $\omega^{(1)}\in\saw\cup\sap$, which is precisely the constraint
    that $A^{(1)}_{1}$ occurs.

    Similarly, the factors of $1-U_{i'j'}(\omega)$ encode the reward
    $(1+\kappa)$ if $\{i',i'+1,j'-1,j'\}$ is a plaquette. Since $0n$
    is \emph{not} compatible, if the first and last edges of
    $\omega^{(1)}$ are parallel this reward is not present. Since
    including this reward gives an upper bound we do so, as it allows
    the right-hand side to be written as
    $\wt_{\kappa}(\omega^{(1)};\emptyset)$.

    For $m\geq 2$ and $1\leq k\leq 2m-1$, the considerations are almost
    exactly the same. The fact that $\omega^{(k)}\in\saw$ follows as
    $i'j'\in\comp_{k}(L)$ if $0\leq i<j\leq n_{k}$ and $j>i+1$ by the
    classification of compatible edges for a lace, and this implies
    $\omega^{(k)}_{i}\neq \omega^{(k)}_{j}$. The definitions of
    $A^{(m)}_{k}$ arise from observing that the set of compatible
    edges have endpoints $i'$ such that $\omega_{i'}$ falls into one
    of the walks comprising $\eta^{(k)}$. We obtain an upper bound by
    including additional rewards $(1+\kappa)$ corresponding to
    incompatible edges, and this recreates the weight
    $\wt_{\kappa}(\omega^{(k)};\eta^{(k)})$.
\end{proof}

\subsection{Diagrammatic bounds}
\label{sec:Lace-bounds}

\Cref{prop:Lace-Edge} leads to an upper bound for
$\pi_{z,\kappa}^{(m)}(x)$ in terms of a sum over collections
$\V{\omega}$ of interacting walks. This will be used to estimate the
size of $\abs{\pi^{(m)}_{z,\kappa}(x)}$ in terms of convolutions of
$G_{z,\kappa}(x)$; the resulting bounds are what are known as
\emph{diagrammatic bounds}. The next definition will be used in
upper-bounding the factors $-U_{ij}(\omega)$ in
\eqref{eq:Connected-Lace}.

\begin{definition}
  Define
  $\vertex_{\kappa}(x) \bydef \indicatorthat{x=o} +
  \kappa\indicatorthat{\norm{x}_{\infty}=1}$.
\end{definition}

In what follows $\kappa = \kappa(\slb)$ will be a function of $\slb$,
chosen such that $\kappa\leq \kappa_{0}(d,\slb)$; in particular
\Cref{prop:ASM} applies when $z\geq z_{0}=(1+\kappa)^{-2(d-1)}$. Let
$H_{z,\kappa}(x) \bydef G_{z,\kappa}(x) - \delta_{x,o}$ be the sum of
contributions to $G_{z,\kappa}(x)$ due to non-trivial walks.

\tikzset{ dot/.style={ circle, inner sep=0pt, minimum size=1.5mm,
    fill=black!100 } }
\tikzset{ bigdot/.style={ circle, inner sep=0pt, minimum size=2mm,
    fill=black!100 } }
\begin{figure}[]
  \centering
    \begin{tikzpicture}[scale=.5]
      \node[dot] (v0) at (0,0) {}; 
      \node[dot] (v1) at (0,2) {};
      \draw[black,thick] (v0) -- (v1); 
      \draw[black,thick,decorate, decoration={zigzag,segment length = 5,
        amplitude=1}] (v0) .. controls (.5,1) .. (v1);
      \node at (v0) [below] {$o$};
      \node at (v1) [above] {$x$};
    \end{tikzpicture}
    \qquad
    \begin{tikzpicture}[scale=.5]
      \node[dot] (v0) at (0,0) {};
      \node[dot] (v1) at (2,0) {};
      \node[dot] (v0p) at (2,2) {};
      \node[dot] (v2) at (4,2) {};
      \draw[black, thick] (v0) -- (v1) -- (v0p) -- (v2);
      \draw[black,thick, decorate, decoration={zigzag,segment
      length = 5, amplitude=1}] (v0) to (v0p); 
      \draw[black,thick, decorate, decoration={zigzag,segment
      length = 5, amplitude=1}] (v1) to (v2); 
      \node at (v0) [below] {$o$};
      \node at (v1) [below] {$y_{2}$};
      \node at (v0p) [above] {$y_{3}$};
      \node at (v2) [above] {$x$};
    \end{tikzpicture}
    \quad
    \begin{tikzpicture}[scale=.5]
    \node[dot] (v0) at (0,0) {};
    \node[dot] (v1) at (2,0) {};
    \node[dot] (v0p) at (2,2) {};
    \node[dot] (v2) at (4,2) {};
    \node[dot] (v1p) at (4,0) {};
    \node[dot] (v3) at (6,0) {};
    \node[dot] (v2p) at (6,2) {};
    \node[dot] (v4) at (8,2) {};
    \node[dot] (v3p) at (8,0) {};
    \node[dot] (v5) at (10,0) {};

   \node at (v1) [below] {$y_{2}$};
    \node at (v0p) [above] {$y_{3}$};
    \node at (v2) [above] {$y_{4}$};
    \node at (v1p) [below] {$y_{5}$};
    \node at (v3) [below] {$y_{6}$};
    \node at (v2p) [above] {$y_{7}$};
    \node at (v4) [above] {$y_{8}$};
    \node at (v3p) [below] {$y_{9}$};

    \node at (v0) [below] {$o$};
    \node at (v5) [below] {$x$};

    \draw[black, thick] (v0) -- (v1) -- (v0p) -- (v2) -- 
    (v1p) -- (v3) -- (v2p) -- (v4) -- (v3p) -- (v5);
    \draw[black,thick, decorate, decoration={zigzag,segment
      length = 5, amplitude=1}] (v0) to (v0p); 
    \draw[black,thick, decorate, decoration={zigzag,segment
      length = 5, amplitude=1}] (v1) to (v1p); 
    \draw[black,thick, decorate, decoration={zigzag,segment
      length = 5, amplitude=1}] (v2) to (v2p);
    \draw[black,thick, decorate, decoration={zigzag,segment
      length = 5, amplitude=1}] (v3) to (v3p);
    \draw[black,thick, decorate, decoration={zigzag,segment
      length = 5, amplitude=1}] (v4) to (v5);
  \end{tikzpicture}
  \caption{Diagrammatic representations of \Cref{prop:DB} in the cases
    $m=1,m=2$ and $m=5$. Wavy lines represent factors
    $\vertex(y_{2j+1}-y_{2j-2})$. For $m\geq 2$ the vertical and first and last
    horizontal straight lines represent factors of $H_{z,\kappa}$. The
    remaining horizontal straight lines represent factors of
    $G_{z,\kappa}$. For $m=1$ the vertical straight line represents $D\ast H_{z,\kappa}$.}
  \label{fig:Diagram-DB}
\end{figure}
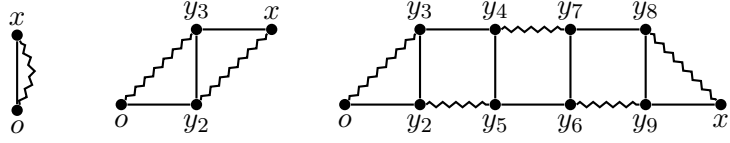

\begin{proposition}
  \label{prop:DB}
  Fix $\kappa\leq \kappa_{0}$, $z_{0}\leq z<z_{c}$, and $x\in
  \Z^{d}$. The following bounds hold. For $m=1$,
  \begin{equation}
    \label{eq:Pi-1-Walks}
    \abs{\pi^{(1)}_{z,\kappa}(x)} 
    \leq (1+\kappa)^{2\nrp} z\vertex_{\kappa}(x) (D\ast H_{z,\kappa})(x).
  \end{equation}
  For $m\geq 2$, let
  $\V{y} \bydef (y_{2},y_{3},\dots,y_{2m-1})\in \Z^{d(2m-2)}$ denote a
  $(2m-2)$-tuple of vertices in $\Z^{d}$. Then
  \begin{align}
    \label{eq:Pi-n-Walks}
    \abs{\pi^{(m)}_{z,\kappa}(x)} 
    &\leq 
       (1+\kappa)^{(4m-2)\nrp} H_{z,\kappa}(y_{2})\vertex_{\kappa}(y_{3}) \\\nonumber
       &\hspace{-20pt}\times \sum_{\V{y}} \Big[ \prod_{j=1}^{m-2}
         \big(H_{z,\kappa}(y_{2j+1}-y_{2j}) 
       G_{z,\kappa}(y_{2j+2}-y_{2j+1})
       \vertex_{\kappa}(y_{2j+3}-y_{2j})\big)\\\nonumber
       &\hspace{40pt}H_{z,\kappa}(y_{2m-1}-y_{2m-2}) H_{z,\kappa}(x-y_{2m-1}) \vertex_{\kappa}(x-y_{2m-2}) \Big]
  \end{align}
  where the sum runs over all $\V{y}\in \Z^{d(2m-2)}$. If $m=2$ the
  product is empty, and hence identically one by definition.
\end{proposition}
A diagrammatic formulation of these upper bounds can be found in
\Cref{fig:Diagram-DB}.
\begin{proof}[Proof of \Cref{prop:DB}]
  We first outline the strategy of the proof. Recall the definition
  \eqref{eq:Pi-n} of $\pi^{(m)}_{z,\kappa}(x)$ as a sum over walks
  $\omega\in\walk(x)$ and laces $L\in\lace^{(m)}$. The proof will use
  the decomposition of a walk $\omega$ into $2m-1$ subwalks
  $\omega^{(i)}$ as given by
  \Cref{def:subwalks}. \Cref{prop:Lace-Edge} gives a formula for the
  weight of the $i^{\text{th}}$ subwalk in terms of the preceding
  subwalks, and by using averaged submultiplicativity
  (\Cref{cor:ASM-E-1}) we will upper bound the sums over subwalks
  $\omega^{(i)}$ by factors of $G_{z,\kappa}$ and $H_{z,\kappa}$. The
  factors of $\vertex_{\kappa}$ will arise when we bound the product
  of $-U_{ij}(\omega)$ over $ij\in L$ in \eqref{eq:Pi-n}.

  Recall the definition \eqref{eq:U-Interaction} of $U_{ij}(\omega)$,
  and observe that for all walks $\omega$
  \begin{equation}
    \label{eq:U-Vertex}
    \abs{U_{ij}(\omega)}\leq \vertex_{\kappa}(\omega_{j}-\omega_{i}),
  \end{equation}
  which follows by considering separately
  $\norm{\omega_{j}-\omega_{i}}_{\infty}$ being $0$, $1$, or at least
  $2$.

  We first consider the case $m=1$. By \eqref{eq:Pi-n},
  \eqref{eq:Connected-Lace}, \Cref{prop:Lace-Edge} and the definition
  \eqref{eq:U-Interaction} of $U_{0n}(\omega)$ we have the upper bound
  \begin{align*}
    \abs{\pi^{(1)}_{z,\kappa}(x)} 
    &\leq  \sum_{n\geq 2} \sum_{\omega\in\walk_{n}(x)} 
      z^{n}\wt_{\kappa}(\omega) \abs{U_{0n}(\omega)}
      \indicatorthat{\omega\in\saw\cup\sap} \\
    &\leq
      \sum_{y\neq o}zD(y) \sum_{n\geq
      1}\sum_{\omega\in\walk_{n}(y,x)}z^{n-1}\wt_{\kappa}(\omega;
      (o,y)) \vertex_{\kappa}(x)
      \indicatorthat{(o,y)\circ\omega\in\saw\cup\sap},
  \end{align*}
  where the second inequality follows from \eqref{eq:U-Vertex} and
  explicitly separating out and summing over the location $y$ of the first
  step of the walk.
  
  Recall $\abs{\omega}=n$ if $\omega\in\walk_{n}$. Let
  $\wt_{z,\kappa}(\omega;\eta) =z^{\abs{\omega}}
  \wt_{\kappa}(\omega;\eta)$, and recall the definition of
  $\bar G^{\eta}_{z,\kappa}(x)$ from below \eqref{eq:Model-BCp}.  The
  sum of  $\wt_{z,\kappa}(\omega;(o,y))
  \indicatorthat{(o,y)\circ\omega\in\saw\cup\sap}$ over
  $\omega\in\walk(y,x)$ of length at least one is at most
  $\bar G^{(o,y)}_{z,\kappa}(y-x)-\delta_{y-x,o}$.  The proposition
  for $m=1$ now follows from \Cref{cor:ASM-E-1} and the translation
  invariance of $H_{z,\kappa}$.
    
  Next we consider $m\geq 2$; the argument is very similar to
  $m=1$. As described in \Cref{sec:Pi-Refined}, if $L\in\lace^{(m)}$
  then $\omega$ is the concatenation of $2m-1$ walks $\omega^{(i)}$,
  $\omega^{(i)}$ has length $n_{i}$, and the vector $\V{n}$ of lengths
  satisfies \eqref{eq:Intervals}. If $\abs{\omega}=n$, distribute the
  factor $z^{n}$ so that there are $n_{i}$ factors associated to the
  walk $\omega^{(i)}$. Then by \Cref{prop:Lace-Edge} and
  \eqref{eq:U-Vertex},
  \begin{align}
    \label{eq:pim-ub1}
    \abs{\pi^{(m)}_{z,\kappa}(x)}
    \leq &\sum_{\V{y}}\sum_{\V{\omega}}
    \vertex_{\kappa}(y_{3})
    \vertex_{\kappa}(x-y_{2m-2})\\\nonumber   
    &\hspace{5mm}
       \prod_{j=2}^{m-1} \vertex_{\kappa}(y_{2j+1}-y_{2j-2})
    \prod_{i=1}^{2m-1} \indicator_{A^{(m)}_{i}}
      \wt_{z,\kappa}(\omega^{(i)};\eta^{(i)}),
  \end{align}
  where the outer sum is over
  $\V{y}=(y_{2},\dots, y_{2m-1})\in \Z^{d(2m-2)}$, the inner sum is
  over $\V{\omega}= (\omega^{(1)}, \dots, \omega^{(2m-1)})$ whose
  lengths $n_{i}$ satisfy \eqref{eq:Intervals} and such that
  $\omega^{(1)}\in\walk(y_{2})$, $\omega^{(i)}\in\walk(y_{i},y_{i+1})$
  for $i=2,\dots,2m-2$, and $\omega^{(2m-1)}\in
  \walk(y_{2m-1},x)$. The sum over laces has been replaced with a sum
  over the possible lengths of the subwalks: see the discussion
  following \eqref{eq:Intervals}.

  We now iteratively sum over the subwalks $\omega^{(i)}$,
  starting with $i=2m-1$. Since $\omega^{(j)}$ is fixed for
  $j=1\, \dots, 2m-2$, the walk $\eta^{(2m-1)}$ is determined, and it
  plays the role of a memory for $\omega^{(2m-1)}$. Temporarily let
  $\tilde n = n_{2m-1}$, $\tilde\omega = \omega^{(2m-1)}$, and
  $\tilde\eta = \eta^{(2m-1)}$, and let
  $\tilde \walk_{\tilde n} = \walk_{\tilde n}(y_{2m-1},x)$.  With these
  definitions we can upper bound the sum over $\omega^{(2m-1)}$ on the
  right-hand side of \eqref{eq:pim-ub1} by
  \begin{equation}
    \label{eq:pim-ub2}
    \mathop{\sum_{\tilde\omega\in\tilde \walk}}_{\tilde n\geq 1}
    \indicator_{A^{(m)}_{2m-1}} 
    \wt_{z,\kappa}(\tilde\omega;\tilde\eta) 
    \leq (1+\kappa)^{2\nrp} 
    \mathop{\sum_{\tilde\omega\in\tilde \walk_{\tilde n}}}_{\tilde n\geq 1}
    \indicatorthat{\tilde\omega\in\saw} 
    \wt_{z,\kappa}(\tilde\omega).
  \end{equation}
  This bound follows from \Cref{cor:ASM-E-1}, as the event
  $A^{(m)}_{2m-1}$ occurs only if $\omega^{(2m-1)}$ does not intersect
  $\eta^{(2m-1)}$ except for the initial, and possibly terminal,
  vertices of $\omega^{(2m-1)}$; we have also used the fact that in
  the sum on the right-hand side of \eqref{eq:pim-ub1} the subwalks
  $\omega^{(i)}$ are all self-avoiding due to the events
  $A^{(m)}_{i}$, and that these events imply $\eta^{(2m-1)}$ is
  self-avoiding or a self-avoiding polygon. The sum on the right-hand side of
  \eqref{eq:pim-ub2} is $H_{z,\kappa}(x-y_{2m-1})$, as there is no
  contribution from zero-step walks due to the constraint
  $\tilde n\geq 1$.

   Repeating this procedure for $2m-2,2m-3,\dots, 1$ gives the
   two-point functions in the upper bounds of the proposition.
   Factors of $H_{z,\kappa}$ arise for edges on which $n_{i}\geq 1$,
   and factors of $G_{z,\kappa}$ for those with $n_{i}\geq 0$; see
   \eqref{eq:Intervals}. The overall factor of $(1+\kappa)^{(4m-2)\nrp}$
   arises as we obtain a factor of $(1+\kappa)^{2\nrp}$ for each factor of
   $H_{z,\kappa}$ and $G_{z,\kappa}$, and there are $2m-1$ of these.
\end{proof}

\subsection{Proof of Gaussian decay}
\label{sec:proof-gaussian-decay}

We will now prove \Cref{thm:High-D} by making use of \Cref{thm:APP-GA}.
Recall that $\mnorm{x} \bydef \max\{\norm{x}_{2},1\}$. 

\begin{proposition}[{\cite[Prop.~1.7]{HvdHS}}]
  \label{prop:HvdHS1.7}
  If $f,g\colon \Z^{d}\to \R$ satisfy $\abs{f(x)}\leq
    \mnorm{x}^{-a}$ and $\abs{g(x)}\leq \mnorm{x}^{-b}$ with $a\geq
    b>0$, there exists $C=C(a,b,d)$ such that
    \begin{equation}
      \abs{(f\ast g)(x)}\leq
      \begin{cases}
        C\mnorm{x}^{-b} & a>d,\\
        C\mnorm{x}^{d-(a+b)} & \textrm{$a<d$ and $a+b>d$}.
      \end{cases}
    \end{equation}
\end{proposition}

\begin{proposition}
  \label{prop:HvdHS1.8}
  Let $d>4$. Suppose $\beta>0$, $\kappa\leq \min\{\kappa_{0},\beta\}$,
  $z_{0}\leq z \leq 2$, and that
  \begin{equation}
    \label{eq:Bound-Assume}
    G_{z,\kappa}(x)\leq \beta \mnorm{x}^{-(d-2)},\qquad x\neq o.
  \end{equation}
  If $\beta\leq \beta_{0}(d)$, then there is a $c=c(d)>0$ such that
  \begin{equation}
    \label{eq:Pi-HvdHS}
    \abs{\Pi_{z,\kappa}(x)}\leq c\beta\indicatorthat{x=o} + 
    \frac{c\beta^{2}}{\mnorm{x}^{3(d-2)}}.
  \end{equation}
\end{proposition}
\begin{proof}
  $G_{z,\kappa}(o)=1$ as only the trivial self-avoiding walk ends at
  the origin, so~\eqref{eq:Bound-Assume} implies
  $G_{z,\kappa}(x)\leq \mnorm{x}^{-(d-2)}$ and
  $H_{z,\kappa}(x)\leq \beta\mnorm{x}^{-(d-2)}$ for all
  $x\in\Z^{d}$. In particular, $\norm{H_{z,\kappa}}_{\infty}\leq\beta$.

  First consider $\pi^{(1)}_{z,\kappa}(x)$. If $x=o$ then
  $\vertex_{\kappa}=1$, so~\eqref{eq:Pi-1-Walks}, $z\leq 2$, and the
  inequality $\norm{f\ast g}_{1}\leq\norm{f}_{\infty}\norm{g}_{1}$ imply
  \begin{equation}
    \label{eq:DB-Pi-1-o}
    \abs{\pi^{(1)}_{z,\kappa}(o)} 
    \leq (1+\kappa)^{2\nrp}z \sum_{y\in\Z^{d}}D(y)H_{z,\kappa}(-y) 
    \leq 2(1+\kappa)^{2\nrp}\beta
  \end{equation}
  since $\sum_{y}D(y)=1$.  If $x\neq o$ then
  $\abs{\vertex_{\kappa}(x)}\leq\kappa\indicatorthat{\norm{x}_{\infty}=1}$,
  and an argument as above shows that
  \begin{align}
    \nonumber
    \abs{\pi^{(1)}_{z,\kappa}(x)} 
    &\leq (1+\kappa)^{2\nrp}\kappa z \indicatorthat{\norm{x}_{\infty}=1}
      \sum_{y\in\Z^{d}}D(y)H_{z,\kappa}(x-y) \\     \label{eq:DB-Pi-1-no}
    &\leq 2(1+\kappa)^{2\nrp}\beta^{2} \indicatorthat{\norm{x}_{\infty}=1}
  \end{align}
  where we have used $\kappa\leq \beta$ and $z\leq 2$.

  Next we consider $\pi^{(m)}_{z,\kappa}(x)$ for $m\geq 2$. The factors
  $\vertex_{\kappa}$ in~\eqref{eq:Pi-n-Walks} imply the collections $\V{y}$
  of vertices that give a non-zero contribution satisfy
  \begin{align}
    \label{eq:vertex-con}
    &\norm{y_{3}}_{\infty}\leq 1, \quad
    \norm{x-y_{2m-2}}_{\infty}\leq 1, \quad \text{and } \\\nonumber
    &\norm{y_{2j+3}-y_{2j}}_{\infty}\leq 1 \text{ for }  j=1,\dots, m-2.
  \end{align}
  Given $\V{y}$ satisfying~\eqref{eq:vertex-con}, define $\V{\rho}$ to
  be the collection of vectors
  \begin{align}
    \label{eq:rho-con}
    &\rho_{1}\bydef y_{3},\quad \rho_{2m+1}\bydef x-y_{2m-2}, \quad
    \text{and} \\\nonumber
   &\rho_{2j+1}\bydef y_{2j+1}-y_{2j-2}, \text{ for } j=1,\dots,m-2.
  \end{align}
  Each $\rho_{j}$ satisfies $\norm{\rho_{j}}_{\infty}\leq 1$. The sum
  over $\V{y}$ in~\eqref{eq:Pi-n-Walks} can be replaced by a sum over
  $y_{i}$ for $i=2,4,\dots, 2m-2$ and a sum over the possible
  $\V{\rho}$. Formally, letting $\pi^{(m)}=\pi^{(m)}_{z,\kappa}$, we
  re-express \eqref{eq:Pi-n-Walks} as
  \begin{equation}
    \label{eq:rho-form}
    \pi^{(m)}(x) \bydef
    (1+\kappa)^{(4m-2)\nrp}\sum_{\V{y}'}\sum_{\V{\rho}}\pi^{(m)}_{\V{y}',\V{\rho}}(x),
  \end{equation}
  where the sum over $\V{y}'$ is over tuples
  $(y_{2},y_{4},\dots, y_{2m-2})$ and this defines the terms
  $\pi^{(m)}_{\V{y}',\V{\rho}}(x)$ as the contributions to
  \eqref{eq:Pi-n-Walks} with the vertices $y_{i}$ determined by
  $\V{y}'$ and $\V{\rho}$; we will shortly give an explicit formula
  for the $\pi^{(m)}_{\V{y}',\V{\rho}}$. We have abused notation
  in~\eqref{eq:rho-form}, but the bold subscripts will ensure that
  $\pi^{(m)}_{\V{y}',\V{\rho}}$ is distinguished from
  $\pi^{(m)}_{z,\kappa}$ in what follows.

  We will now show the proposition follows from the estimate
  \begin{equation}
    \label{eq:DB-Unif-x1}
    \sum_{\V{y}'} \pi^{(m)}_{\V{y}',\V{\rho}}(x) \leq
    \beta^{m}C^{m}\mnorm{x}^{-3(d-2)}, \qquad m\geq 2,
  \end{equation}
  for a constant $C>0$ independent of $\beta$.  \Cref{eq:DB-Unif-x1}
  is uniform in $\V{\rho}$, so with~\eqref{eq:rho-form} it
  implies
  \begin{equation}
    \label{eq:DB-n-Bound}
    \abs{\pi^{(m)}_{z,\kappa}(x)} \leq 
    (C'\beta)^{m}\mnorm{x}^{-3(d-2)}, \qquad m\geq 2,
  \end{equation}
  where $C'$ can be taken to be
  $C (1+(3^{d}-1)\kappa) (1+\kappa)^{2\nrp}$. The factor of
  $(1+\kappa)^{2\nrp}$ is from the prefactor in
  \eqref{eq:rho-form}.  The factor of $(1+(3^{d}-1)\kappa)$ arises as
  (i) each $\rho_{i}$ has $3^{d}-1$ non-zero possibilities, (ii) each
  $i$ with $\norm{\rho_{i}}_{\infty}=1$ carries a factor of $\kappa$
  from $\vertex_{\kappa}(\rho_{i})$, and (iii)
  $\vertex_{\kappa}(o)=1$. Summing~\eqref{eq:DB-n-Bound} over
  $m\geq 2$ and combining it with the bounds~\eqref{eq:DB-Pi-1-o}
  and~\eqref{eq:DB-Pi-1-no} for $m=1$ implies the proposition. The
  dependence of $\beta$ in the proposition is on $d$ alone because
  $\kappa\leq\min\{\kappa_{0},\beta\}$ by hypothesis, so $C'$ depends
  only on the dimension $d$. 

  The remainder of the proof establishes~\eqref{eq:DB-Unif-x1}, and
  for this we need an explicit formula for
  $\pi^{(m)}_{\V{y}',\V{\rho}}(x)$. Fix $\V{\rho}$, let
  $H=H_{z,\kappa}$, $G=G_{z,\kappa}$, $\vertex=\vertex_{\kappa}$, and
  $y_{0}=o$.  Recall that $\trans_{a}f(x)=f(x-a)$ for
  $f\colon \Z^{d}\to \R$ and $x,a\in\Z^{d}$. This yields
  \begin{align}
    \nonumber \pi^{(m)}_{\V{y}',\V{\rho}}(x)
    &\bydef
      \sum_{\V{y}'} H(y_{2}) \prod_{j=1}^{m-2}\ob{
      \trans_{-\rho_{2j+1}}H(y_{2j-2}-y_{2j})
      \trans_{\rho_{2j+1}}G(y_{2j+2}-y_{2j-2})}
    \\\label{eq:Pre-Conv-Form}&\hspace{20mm}
                                \trans_{\rho_{2m-1}}H(y_{2m-4}-y_{2m-2})
                                \trans_{-\rho_{2m-1}}H(x-y_{2m-4})
  \end{align}
  where the sum is over $y_{2},\dots, y_{2m-2}$.

  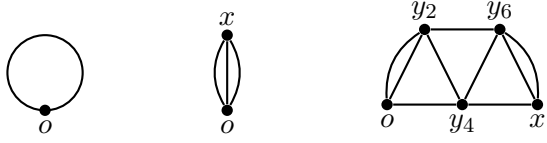
\begin{figure}[]
    \centering
        \begin{tikzpicture}[scale=.5]
      \node[dot] (v0) at (0,0) {}; 
      \draw[black,thick] (0,1) circle (1);
      \node at (v0) [below] {$o$};
    \end{tikzpicture}
    \qquad\qquad
    \begin{tikzpicture}[scale=.5]
      \node[dot] (v0) at (0,0) {};
      \node[dot] (v1) at (0,2) {};
      \draw[black, thick] (v0) -- (v1);
      \draw[black,thick] (v0) edge [bend right = 30] (v1);
      \draw[black,thick] (v0) edge [bend left = 30] (v1);
      \node at (v0) [below] {$o$};
      \node at (v1) [above] {$x$};
    \end{tikzpicture}
    \qquad\qquad
    \begin{tikzpicture}[scale=.5]
    \node[dot] (v0) at (0,0) {};
    \node[dot] (v1) at (1,2) {};
    \node[dot] (v2) at (2,0) {};
    \node[dot] (v3) at (3,2) {};
    \node[dot] (v4) at (4,0) {};

    \node at (v0) [below] {$o$};
   \node at (v1) [above] {$y_{2}$};
    \node at (v2) [below] {$y_{4}$};
    \node at (v3) [above] {$y_{6}$};
    \node at (v4) [below] {$x$};

    \draw[black, thick] (v0) -- (v1) edge[bend right=30] (v0);
    \draw[black,thick] (v0) -- (v2) -- (v1) -- 
    (v3) -- (v2) -- (v4) -- (v3) edge[bend left=30] (v4);
  \end{tikzpicture}
    \caption{Diagrammatic representation of \Cref{eq:Pre-Conv-Form}
      when $m=1,2,5$. This is precisely the form of the upper bounds for
      self-avoiding walk, i.e., $\kappa=0$. Horizontal lines
      represent factors of $G_{z,\kappa}$ or its translates, and the
      remaining lines represent factors of $H_{z,\kappa}$.}
    \label{fig:Pre-Conv}
  \end{figure}

  Diagrammatically, see~\Cref{fig:Pre-Conv}, \Cref{eq:Pre-Conv-Form}
  has exactly the form of a self-avoiding walk diagrammatic
  bound~\cite{Slade}, but where the two-point functions $G,H$ have
  been replaced with their translates. Our estimates for $G$ and $H$
  imply there is an $\con = \con(d)$ such that
  \begin{align}
    \label{eq:G-Trans}
    \abs{\trans_{\rho}G_{z,\kappa}(x)} &\leq \con(d)\mnorm{x}^{-(d-2)}
    \\
    \label{eq:H-Trans}
    \abs{\trans_{\rho}H_{z,\kappa}(x)} &\leq \con(d)\beta\mnorm{x}^{-(d-2)}
  \end{align}
  since $\mnorm{x+\rho}/\mnorm{x}$ is uniformly bounded above when
  $\norm{\rho}_{\infty}\leq 1$. Thus, the two-point functions
  $\trans_{\rho}G$ and $\trans_{\rho}H$ satisfy, up to a constant
  depending only on $d$, the same estimates as do $G$ and $H$. 

  The remainder of the proof is standard in lace expansion analyses,
  and hence we will be somewhat brief. See,
  e.g.,~\cite[proof of Prop.~1.8(a)]{HvdHS} for more details.

  Define $\tilde{G}$ and $\tilde{H}$ to be the upper bounds on $G$ and $H$
  given by the right-hand sides of \Cref{eq:G-Trans,eq:H-Trans}. Let 
  \begin{align*}
    A(u,v,x,y)  &\bydef \tilde{H}(v-u)\tilde{G}(y-u)\indicatorthat{v=x},\\
    M^{(2)}(x,y) &\bydef \tilde{H}(x)^{2}\tilde{G}(y), \\
    M^{(m)}(x,y) &\bydef \sum_{u,v\in\Z^{d}}M^{(m-1)}(u,v)A(u,v,x,y), \qquad
    m\geq 3.
  \end{align*}
  With these definitions, we obtain
  \begin{equation}
    \label{eq:Conv-Bound}
    \sum_{\V{y}'}\pi^{(m)}_{\V{y}',\V{\rho}}(x) \leq M^{(m)}(x,x), \qquad m\geq 2,
  \end{equation}
  where in the case $m=2$ we have degraded the bound slightly by using
  the estimate $H\leq G$. By \eqref{eq:G-Trans} and \eqref{eq:H-Trans}
  there is a constant $c'=c'(d)>0$ such that
  \begin{equation}
    \label{eq:A-Bound}
    A(u,v,x,y) \leq
    \frac{c'\beta}{\mnorm{v-u}^{d-2}\mnorm{y-u}^{d-2}}
    \indicatorthat{v=x}.
  \end{equation}
  Define
  \begin{equation*}
    \bar{\cc S}
    \bydef \sup_{x\in\Z^{d}}\sum_{y\in \Z^{d}} \frac{1}{\mnorm{y}^{d-2}{\mnorm{x-y}^{d-2}}}.
  \end{equation*}
  When $d>4$, $\bar{\cc S}$ is finite by an elementary convolution
  estimate~\cite[Proposition~1.7]{HvdHS}.  By an induction on $m$
  using~\eqref{eq:A-Bound} it can be shown~\cite[p.\
    381-382]{HvdHS} that this implies there is
  a $C=C(d)$ such that
  \begin{equation}
    \label{eq:M-Bound}
    M^{(m)}(x,y)\leq (c'\beta)^{m}(C\bar{\cc S})^{m-2}
    \frac{1}{\mnorm{x}^{2(d-2)}\mnorm{y}^{d-2}}, \qquad m\geq 2,
  \end{equation}
  which proves \Cref{eq:DB-Unif-x1}.
\end{proof}

\begin{proof}[Proof of \Cref{thm:High-D}]
  To prove \Cref{thm:High-D}, it suffices to verify that there is a
  $\kappa_{0}$ such that, if $\kappa\leq \kappa_{0}$, the hypotheses of
  \Cref{sec:HP} on $D$, $G_{z,\kappa}$, and $\Pi_{z,\kappa}$ are
  satisfied.

  \Cref{hyp:SD} is trivially satisfied. \Cref{hyp:CP} is satisfied for
  $\kappa\leq \kappa_{0}(L)$ for some $\kappa_{0}(L)$ by
  \Cref{thm:CC-Exists-Intro} which ensures the critical point exists,
  and \Cref{prop:Gamma-LB}, which ensures the divergence of the
  susceptibility.

  We now verify the monotonicity hypothesis, (ii), and (iii) of
  \Cref{hyp:SC-Decay}. Since $G_{z,\kappa}(x)$ is an
  absolutely convergent power series with positive coefficients when
  $z<z_{c}$, it is monotone and continuous for $z<z_{c}$. The
    exponential decay hypothesis is provided by \Cref{lem:expdecay}.

  To verify (i) of \Cref{hyp:SC-Decay}, let
  $z_{0}=(1+\kappa)^{-2(d-1)}$. When $z\leq z_{0}$,
  \begin{align}
    \label{eq:Bootstrap-Start}
   G_{z,\kappa}(x) 
       &= \sum_{n} \sum_{\omega\in\saw_{n}(x)}
       z^{n}\wt_{\kappa}(\omega) 
                   \leq \sum_{n} \sum_{\omega\in\saw_{n}(x)}
      (1+\kappa)^{-2(d-1)n}\wt_{\kappa}(\omega) \\
    \nonumber
   &\leq \sum_{n} \sum_{\omega\in\saw_{n}(x)}
      \P_{n}(\omega) 
       = G_{1,0}(x).
  \end{align}
  This implies $G_{z_{0},\kappa}(x)\leq \srwtwo_{1}(x)$, as
  $\srwtwo_{1}(x)$ is clearly an upper bound for the SAW two-point
  function. 

  For any $D$, \Cref{hyp:Pi-Bound} follows for $G_{z,\kappa}$ by
  \Cref{prop:HvdHS1.8} when $\kappa$ is small enough, with $\beta_{0}$
  uniform in $\kappa$.  Thus, for $\kappa\leq \kappa_{0}(L_{0})$,
with  $L_{0}$ the constant of \Cref{thm:APP-GA}, we can apply
  \Cref{thm:APP-GA} by the discussion of \Cref{sec:other}. This proves
  the theorem.
\end{proof}

\subsection{Acknowledgments}
\label{sec:acknowledgements}

The authors would like to thank both referees for their
critiques and comments, which have lead to a significantly improved
article. T.H.\ would like to thank Gordon Slade and Remco van der
Hofstad for encouraging discussions. A.H.\ is supported by NSF grant
DMS-1512908. The majority of this work was carried out while T.H.\ was
supported by an NSERC postdoctoral fellowship at UC Berkeley;
additional support was provided by EPSRC grant EP/P003656/1.

\appendix

\section{Gaussian Asymptotics}
\label{app:GA}

This appendix reviews \cite[Theorem~1.2]{HvdHS}, which derives Gaussian
asymptotics for critical two-point functions. Our motivation is that
the presentation in~\cite{HvdHS} is, at places, dependent on the
particular models being studied. The proofs, however, apply
essentially verbatim to other models. Our review axiomatizes
sufficient assumptions for models similar to self-avoiding walk. We
indicate where these assumptions are used in proofs, but omit the
portions of the proofs that purely replicate~\cite{HvdHS}. We emphasise
that the result and techniques are those of~\cite{HvdHS}, and our
presentation is primarily for the benefit of the reader who is not
familiar with~\cite{HvdHS}.

\subsection{Setup}
\label{sec:Setup}

Let $\R_{\geq 0}$ denote the non-negative reals. For
$z\in\R_{\geq 0}$, $G_{z}\colon \Z^{d}\to\R_{\geq 0}$,
$\HHSpi_{z}\colon\Z^{d}\to\R$, and $D$ a probability distribution on
$\Z^{d}$, we consider the convolution equation
\begin{equation}
  \label{eq:LSD}
  G_{z}(x)= \delta_{o,x} + \HHSpi_{z}(x) +
  (zD\ast(\delta+ \HHSpi_{z})\ast G_{z})(x).
\end{equation}
We will further assume that $G_{z}$, $\HHSpi_{z}$, and $D$ are all
$\Z^{d}$-symmetric, and that $G_{z}(x)$ is a power series in
  $z$ with non-negative coefficients. We will see in \Cref{sec:other}  that the analysis
of~\eqref{eq:LSD} also applies to the convolution equation derived for
$\kappa$-ASAW in the main body of the text.

The \emph{critical point $z_{c}$} is
$z_{c} = \sup\{z\in\R_{\geq 0} \mid \chi(z)<\infty\}$,  
where the \emph{susceptibility $\chi(z)$} is defined by
\begin{equation}
  \label{eq:Susc}
  \chi(z) \bydef \sum_{x\in\Z^{d}}G_{z}(x).
\end{equation}

\subsection{Hypotheses and Theorem}
\label{sec:HP}

\begin{hypothesis}
  \label{hyp:SD}
  Assume that $D$ is a spread-out step distribution as defined in
  \Cref{def:SD}.
\end{hypothesis}
Let $X_{n}$ be a discrete time simple random walk with step
distribution $D$. Let
$\sigma^{2} = \sum_{x\in\Z^{d}}D(x)\norm{x}^{2}_{2}$. Note that
$\sigma^{2}$ is comparable to the spread-out 
parameter~$L^{2}$. The
\emph{non-interacting two-point function $\srwtwo_{\mu}$} is defined by
\begin{equation}
  \label{eq:SRW-2PT}
  \srwtwo_{\mu}(x) \bydef \sum_{n=0}^{\infty}\mu^{n}\PP_{0}\cb{X_{n}=x}.
\end{equation}

An important consequence of the form of $D$ is the following
proposition. Let $a_{d} \bydef \frac{d\Gamma(d/2-1)}{2\pi^{d/2}}$, 
where $\Gamma$ is Euler's gamma function.
\begin{proposition}[{\cite[Prop.~1.6]{HvdHS}}]
  \label{prop:srw}
  Suppose $d>2$ and \Cref{hyp:SD} holds. For $L$ sufficiently large,
  $\alpha>0$, $\mu\leq 1$, and $x\in\Z^{d}$,
  \begin{align}
    \label{eq:SRW-Decay}
    \srwtwo_{\mu}(x) &\leq \delta_{o,x} +
                 O\left(\frac{1}{L^{2-\alpha}\mnorm{x}^{d-2}}\right)\\
    \label{eq:SRW-Decay-2}
    \srwtwo_{1}(x) &= \frac{a_{d}}{\sigma^{2}}\frac{1}{\mnorm{x}^{d-2}}+
               O\left(\frac{1}{\mnorm{x}^{d-\alpha}}\right).
  \end{align}
  The implicit constants may depend on $\alpha$, but not on $L$.
\end{proposition}
Note that, for fixed $d$, the leading coefficient
in~\eqref{eq:SRW-Decay-2} is proportional to $L^{-2}$.
The next two hypotheses deal with the critical point
and behaviour of $G_{z}$ for $z_{0}\leq z<z_{c}$, where $z_{0}>0$ is a
chosen value of the parameter $z$.

\begin{hypothesis}
  \label{hyp:CP}
  The critical point $z_{c}$ satisfies $z_{0}<z_{c}<\infty$. The
  susceptibility specified by \eqref{eq:Susc} diverges as the critical
  point is approached from below:
  $\lim_{z\uparrow z_{c}}\chi(z)=\infty$.
\end{hypothesis}

\begin{hypothesis}
  \label{hyp:SC-Decay}
  $G_{z}$ is well-defined, not identically zero, and monotone
    increasing in $z$.
  For $z_{0}\leq z<z_{c}$ and for each $x\in \Z^{d}$, 
  \begin{enumerate}
  \item $G_{z_{0}}(x)\leq \srwtwo_{1}(x)$,
  \item $G_{z}(x)$ is continuous for $z\in\co{z_{0},z_{c}}$, and
  \item for $t>0$ and $z\in\co{z_{0},z_{c}-t}$ there are constants
    $c(t),C(t)>0$ such that
    \begin{equation}
      \label{eq:subcrit-mass}
      G_{z}(x)\leq C(t)e^{-c(t)\mnorm{x}}.
    \end{equation}
  \end{enumerate}
\end{hypothesis}

The most substantial hypothesis is the next one.
\begin{hypothesis}
  \label{hyp:Pi-Bound}
  Assume
  \begin{equation}
    \label{eq:xIRB-H}
    G_{z}(x)\leq \beta \mnorm{x}^{-d+2}, \qquad x\neq o.
  \end{equation}
  Suppose also that $z_{0}\leq z\leq 2$.  If $\beta<\beta_{0}$, there
  is a constant $c=c(d)>0$ such that
  \begin{equation}
    \label{eq:PIb-H}
    \abs{\HHSpi_{z}(x)} \leq c\beta\delta_{o,x} +
    \frac{c\beta^{2}}{\mnorm{x}^{3(d-2)}}.
  \end{equation}
\end{hypothesis}

\begin{theorem}[{\cite[Theorem~1.2]{HvdHS}}]
  \label{thm:APP-GA}
  Assume $D$, $G_{z}$, and $\HHSpi_{z}$ satisfy the hypotheses of
  \Cref{sec:HP}.  Choose $0<\alpha<2$. Let $\beta_{0}$ be the constant
  of~\Cref{hyp:Pi-Bound}.
  
  There is an $L_{0}(d,\alpha,\beta_{0})$ such that, for $L\geq L_{0}$,
  the function $G_{z_{c}}\colon\Z^{d}\to\R$ is well-defined, and there
  is an $A>0$ such that
  \begin{equation}
    \label{eq:GA}
    G_{z_{c}}(x) \sim \frac{a_{d}A}{\sigma^{2}\mnorm{x}^{2-d}} 
    \left(1+O\left(\frac{L^{2}}{\mnorm{x}^{2-\alpha}}\right)\right).
  \end{equation}
  The implicit constants are uniform in $x$ and $L$. The values of
  $z_{c}$ and $A$ are $1+O(L^{\alpha-2})$.
\end{theorem}

\subsection{Proof}
\label{sec:proof}

The next proposition is the heart of the analysis. In what follows we
assume the hypotheses of \Cref{thm:APP-GA}; in
particular, $\beta_{0}$ is given.

\begin{proposition}
  \label{prop:xIRB}
  Fix $\alpha>0$. There is an $L_{0}=L_{0}(\beta_{0},d,\alpha,z_{0})$ such
  that, for $L\geq L_{0}$,
  \begin{equation}
    \label{eq:xIRB}
    G_{z_{c}}(x) \leq
    \frac{\mathrm{const}}{L^{2-\alpha}\mnorm{x}^{d-2}}, \qquad x\neq o,
  \end{equation}
  and $z_{c}\leq 1 + O(L^{-2+\alpha})$.
\end{proposition}

\begin{lemma}[Lemma~2.1~\cite{HvdHS}]
  \label{lem:BS}
  Let $f\colon[z_{1},z_{c})\to\R$, and $a\in(0,1)$. Suppose
  \begin{enumerate}
  \item\label{bs1} $f$ is continuous on $[z_{1},z_{c})$,
  \item\label{bs2} $f(z_{1})\leq a$, and 
  \item\label{bs3} for $z\in[z_{1},z_{c})$ the inequality $f(z)\leq
    1$ implies the inequality $f(z)\leq a$. 
  \end{enumerate}
  Then $f(z)\leq a$ for all $z\in[z_{1},z_{c})$.
\end{lemma}

\begin{proof}[Proof of~\Cref{prop:xIRB}]
  The proof is essentially that in~\cite{HvdHS}. We present the steps in
  which our hypotheses, as opposed to model-specific facts, are used.

  Note that it suffices to prove that~\eqref{eq:xIRB} holds for
  $\alpha < \frac{1}{2}$, as the right-hand side is increasing in
  $\alpha$. By \Cref{hyp:SC-Decay} and the monotone convergence
  theorem, it is enough to prove this for all $z_{0}<z<z_{c}$.
  
  Let $K$ be the optimal constant for the error bound
  in~\Cref{prop:srw}:
  \begin{equation*}
    K = \sup_{L\geq 1,x\neq o} L^{2-\alpha}\mnorm{x}^{d-2}\srwtwo_{1}(x),
  \end{equation*}
  and note $K$ is finite by \eqref{eq:SRW-Decay}.
  Define
  \begin{equation*}
    g_{x}(z) = (2K)^{-1}L^{2-\alpha}\mnorm{x}^{d-2}G_{z}(x),
  \end{equation*}
  and let $g(z) = \sup_{x\neq o}g_{x}(z)$. To prove
  \eqref{eq:xIRB}, we will use \Cref{lem:BS} with
  $f(z) = \max\{g(z),\frac{z}{2z_{0}}\}$, $z_{1}=z_{0}$, and
  $a\in\ob{\frac{1}{2},1}$ arbitrary. The claim that
  $z_{c}=1+O(L^{-2+\alpha})$ will be established in the course of the
  argument.  \vspace{.5em}
  \begin{claim}{}
    Hypothesis~(i) of \Cref{lem:BS} holds. 
  \end{claim}
  \begin{claimproof}
    For $x\in \Z^{d}$, $g_{x}(z)$ is continuous on $[z_{0},z_{c})$ by
    \Cref{hyp:SC-Decay}. It suffices to show $\sup_{x\neq o}g_{x}(z)$
    is continuous on $[z_{0},z_{c}-t)$ for arbitrarily small $t>0$.
    
    Fix $t>0$, and let $z\in[z_{0},z_{c}-t)$. By
    \Cref{hyp:SC-Decay}, $g_{x}(z)$ decays exponentially in
    $\norm{x}_{2}$ with decay rate independent of $z$. 
    Therefore, $\sum_{x\in\Z^{d}}g_{x}(z)$ converges exponentially
    fast with rate independent of $z$. 
    It follows that the supremum of $g_{x}(z)$ occurs on $B_{R}(o)$,
    the ball of radius $R$ about the origin, for some $R=R(L)>0$. This
    proves $\sup_{x\neq o}g_{x}(z)$ is a continuous function of
      $z\in[z_{0},z_{c}-t)$ since 
      the supremum of a finite set of continuous
    functions is continuous.
  \end{claimproof}
  \vspace{.5em}
  \begin{claim}{}
    Hypothesis~(ii) of \Cref{lem:BS} holds. 
  \end{claim}
  \begin{claimproof}
    By \Cref{hyp:SC-Decay} and the definition
    of $K$, $g_{x}(z_{0})\leq \frac{1}{2}$ for all $x$. Since
    $a>\frac{1}{2}$, this proves the claim.
  \end{claimproof}
  \vspace{.5em}
  \begin{claim}{}
    Hypothesis~(iii) of \Cref{lem:BS} holds. 
  \end{claim}
  \begin{claimproof}
    Fix $z_{0}<z<z_{c}$ and suppose $f(z)\leq 1$. Then $z$ is at most
    $2z_{0}$, and
    \begin{equation}
      \label{eq:xIRB-P}
      G_{z}(x)\leq 2z_{0}KL^{-2+\alpha}\mnorm{x}^{2-d},\qquad x\neq o.
    \end{equation}
    Let $\beta = 2z_{0}KL^{-2+\alpha}$. By \Cref{hyp:Pi-Bound}, when
    $L^{-2+\alpha}$ is sufficiently small there is a $c>0$ such that
    \begin{equation}
      \label{eq:xIRB-Pi}
      \abs{\HHSpi_{z}(x)}\leq c\beta\delta_{o,x} +
                          c\beta^{2}\mnorm{x}^{-3(d-2)}
      \leq \frac{c\beta}{\mnorm{x}^{3(d-2)}}.
    \end{equation}
    By \Cref{hyp:SC-Decay}, $G_{z}$ is not identically zero. Thus $\chi(z)>0$, and
    the sum of~\eqref{eq:LSD} over all $x\in\Z^{d}$ can be rearranged
    to give
    \begin{equation}
      \label{eq:CP-ID}
      \chi(z) =
      \frac{1+\sum_{x}\HHSpi_{z}(x)}{1-z-z\sum_{x}\HHSpi_{z}(x)}>0.
    \end{equation}
    By~\eqref{eq:xIRB-Pi}, $\norm{\HHSpi_{z}(x)}_{1}<1$ for $L$ large
    enough. This implies the numerator, and hence the denominator,
    of~\eqref{eq:CP-ID} is strictly positive. Since $f(z)\leq 1$, this
    implies that
    \begin{equation}
      \label{eq:CP-Bound}
      z < 1-z\sum_{x\in\Z^{d}}\HHSpi_{z}(x) \leq 1 + O(z_{0}L^{-2+\alpha}).
    \end{equation}
    Thus $\frac{z}{2}$ is bounded above by $a$ for
    $a\in\ob{\frac{1}{2},1}$, provided that $L$ is large
    enough. 

    What remains is to prove $g(z)\leq a$ for $a\in\ob{\frac{1}{2},1}$
    when $L$ is large enough. This exactly follows the presentation
    in~\cite[p.364]{HvdHS}, and hence we omit it.
  \end{claimproof}
  \vspace{.5em}

  By \Cref{hyp:SC-Decay} this proves the desired bounds, as we have
    proven that $f(z)\leq a$ for $z_{0}\leq z<z_{c}$. The bound on
    $z_{c}$ follows from~\eqref{eq:CP-Bound}, which holds as it was
    derived under the hypothesis that $f(z)\leq 1$.
\end{proof}

\begin{proof}[Proof of~\Cref{thm:APP-GA}]
  This follows~\cite[Theorem~1.2]{HvdHS}. The only model specific step
  in the cited proof is showing that an auxiliary parameter $\mu_{z}$
  increases to $\mu_{z_{c}}=1$ as $z\uparrow z_{c}$. We define this
  parameter below and show that it takes the desired value by
  \Cref{hyp:CP}.

  By~\eqref{eq:xIRB-Pi}, $\HHSpi_{z}(x)$ has a finite second moment
  when $L$ is large enough.  It therefore makes sense to define
  \begin{align}
    \label{eq:lambda-def}
    \lambda_{z} &=
    \frac{1}{1+z\sigma^{-2}\sum_{x}\norm{x}^{2}_{2} \HHSpi_{z}(x)}, \\
    \label{eq:mu-def}
    \mu_{z} &= 1-\lambda_{z}\left(1-z-z\sum_{x}\HHSpi_{z}(x)\right).
  \end{align}  
  \Cref{eq:xIRB-Pi} implies $\lambda_{z}\to 1$ as
  $L\to\infty$ uniformly  in $z\in\cb{z,z_{c}}$.
  By \Cref{eq:CP-ID} and \Cref{hyp:CP}, as $z\uparrow z_{c}$, the
  quantity in brackets in~\eqref{eq:mu-def} tends to zero. Thus,
  $\mu_{z_{c}}\uparrow 1$ as $z\uparrow z_{c}$.
\end{proof}

\subsection{Other convolution equations}
\label{sec:other}

Consider the equation
\begin{equation}
  \label{eq:Lace-Ours}
  G_{z} = \delta +  z (D\ast G_{z}) +
  (\Pi_{z}\ast G_{z}).
\end{equation}
If $\Pi$ satisfies \Cref{hyp:Pi-Bound}, it is possible to
manipulate~\eqref{eq:Lace-Ours} into the form~\eqref{eq:LSD}. To see
this, rewrite~\eqref{eq:Lace-Ours} as
\begin{align*}
  G &= \delta + \Pi + zD\ast(\delta + \Pi)\ast G -\Pi\ast(\delta + 
      zD\ast G - G) \\
    &= \delta + \Pi + zD\ast(\delta + \Pi)\ast G + \Pi\ast\Pi\ast G,
\end{align*}
where, in the second equality, we have used~\eqref{eq:Lace-Ours} to
rewrite the term in parentheses, and the subscripts $z$ have been
omitted. Rewriting the last factor of $G$ using~\eqref{eq:Lace-Ours}
yields
\begin{equation*}
  G = \delta + \Pi + \Pi^{\ast 2} + zD\ast (\delta + \Pi  + \Pi^{\ast 2})\ast G +
  \Pi^{\ast 3}\ast G,
\end{equation*}
where $A^{\ast k}$ is the $k$-fold autoconvolution of $A$.  Iterating
this yields~\eqref{eq:LSD} with
\begin{equation}
  \label{eq:newpi}
  \HHSpi_{z} = \sum_{k\geq 1}\Pi^{\ast k},
\end{equation}
since $\lim_{n\to\infty}\Pi^{\ast n}=0$ under the assumption that~$\Pi$
satisfies \Cref{hyp:Pi-Bound}.  Finally, \cite[Proposition~1.7]{HvdHS}
implies that, if $\Pi_{z}$ satisfies \Cref{hyp:Pi-Bound}, then
$\HHSpi_{z}$ defined by~\eqref{eq:newpi} satisfies
\Cref{hyp:Pi-Bound}, for possibly different constants. The change in
constants depends only on $d$. See~\cite[Section~4.1]{HvdHS} for a
further discussion of this point. Thus to apply
\Cref{thm:APP-GA} to the convolution
equation~\eqref{eq:Lace-Ours}, it suffices to verify the hypotheses of
\Cref{sec:HP} for $G_{z}$, $D$, and $\Pi$.

\end{document}